%% file: main.tex
\documentclass[12pt]{amsart}
\usepackage[T1]{fontenc}
\usepackage[utf8]{inputenc}
\usepackage[margin=1in]{geometry}
\usepackage{mdframed}
\usepackage[colorinlistoftodos]{todonotes}
\usepackage{tikz}
\usetikzlibrary{matrix}
\usepackage{float}
\usepackage{amssymb}
\usepackage{amsmath}
\usepackage[title]{appendix}
\usepackage{amsthm}
\usepackage{amsfonts}
\usepackage{tikz-cd}
\usepackage{comment}
\usepackage{mathtools}
\usepackage[colorlinks=true,
            linkcolor=black,
            citecolor=black,
            urlcolor=black]{hyperref}
\usepackage{mathrsfs}
\usepackage{enumitem}
\usepackage{tabularx}
\usepackage{todonotes}
\DeclarePairedDelimiter\abs{\lvert}{\rvert}%
\newtheorem{theorem}{Theorem}[section]
\newtheorem{lemma}[theorem]{Lemma}

\newtheorem{definition}[theorem]{Definition}
\newtheorem{example}[theorem]{Example}
\newtheorem{proposition}[theorem]{Proposition}
\newtheorem{notation}[theorem]{Notation}

\newtheorem{corollary}[theorem]{Corollary}

\theoremstyle{plain} % just in case the style had changed
\newcommand{\thistheoremname}{}
\newtheorem{genericthm}[theorem]{\thistheoremname}

\newtheorem*{genericthm*}{\thistheoremname}
\newenvironment{namedthm*}[1]
  {\renewcommand{\thistheoremname}{#1}%
   \begin{genericthm*}}
  {\end{genericthm*}}
\newcommand{\Z}{\mathbb{Z}}
\newcommand{\Q}{\mathbb{Q}}

\newcommand{\C}{\mathbb{C}}

\newcommand{\E}{\mathbb{E}}
\renewcommand{\S}{\mathbb{S}}
\renewcommand{\P}{\mathbb{P}}
\newcommand{\F}{\mathbb{F}}

\DeclareMathOperator{\aug}{\mathrm{aug}}
\DeclareMathOperator{\GL}{\text{GL}}

\DeclareMathOperator{\id}{\text{id}}

\DeclareMathOperator{\Mod}{\text{Mod}}
\DeclareMathOperator{\cofib}{\mathrm{cofib}}
\DeclareMathOperator*{\colim}{\mathrm{colim}}
\DeclareMathOperator{\fib}{\mathrm{fib}}
\DeclareMathOperator{\inv}{\mathrm{inv}}

\DeclareMathOperator{\Spev}{\mathrm{Spev}}
\DeclareMathOperator{\fil}{\mathrm{fil}}
\DeclareMathOperator{\G}{\mathbb{G}}
\DeclareMathOperator{\gr}{\mathrm{gr}}
\DeclareMathOperator{\mot}{\mathrm{mot}}
\DeclareMathOperator{\ev}{\mathrm{ev}}
\DeclareMathOperator{\QCoh}{\mathrm{QCoh}}
\DeclareMathOperator{\Sp}{Sp}
\DeclareMathOperator{\ku}{ku}
\DeclareMathOperator{\MU}{\mathrm{MU}}

\DeclareMathOperator{\Hom}{\mathrm{Hom}}

\DeclareMathOperator{\BP}{\mathrm{BP}}
\DeclareMathOperator{\Map}{\text{Map}}
\DeclareMathOperator{\BGL}{\text{BGL}}
\DeclareMathOperator{\BU}{\mathrm{BU}}

\DeclareMathOperator{\SU}{\mathrm{SU}}
\DeclareMathOperator{\Alg}{\mathrm{Alg}}
\DeclareMathOperator{\CoAlg}{\mathrm{CoAlg}}
\DeclareMathOperator{\Modl}{\mathrm{Mod}}
\DeclareMathOperator{\op}{\mathrm{op}}

\DeclareMathOperator{\Spf}{\mathrm{Spf}}
\DeclareMathOperator{\U}{\mathrm{U}}
\DeclareMathOperator{\BSU}{\mathrm{BSU}}
\DeclareMathOperator{\BBSU}{\mathrm{BBSU}}

\DeclareMathOperator{\THH}{\mathrm{THH}}
\DeclareMathOperator{\W}{\mathrm{W}}
\DeclareMathOperator{\Spec}{\mathrm{Spec}}
\DeclareMathOperator{\Spc}{\mathrm{Spc}}

\DeclareMathOperator{\conj}{\mathrm{conj}}

\DeclareMathOperator{\Pic}{\mathrm{Pic}}

\title[Generalized Thom Spectra, Conjugate Stacks, and $\mathrm{BP}\langle n \rangle$]{Generalized Thom Spectra, Conjugate Stacks, and truncated Brown-Peterson Spectra}
\author{Preston Cranford}

\date{version: 2026-Oct-12 v1}
\begin{document}
\begin{abstract}
  We propose a definition of conjugate stack for a class of suitable $\E_3$-$\mathrm{MU}$-algebras that includes the truncated Brown-Peterson spectra $\mathrm{BP}\langle n \rangle$ for all heights $n$. Then we calculate part of $\mathrm{BP}\langle n \rangle^{\mathrm{conj}}$ as a quotient stack in terms of Witt vectors. To perform this computation we study generalized Thom spectra and their relative Topological Hochschild Homology. We exhibit the Eilenberg-Mac Lane spectrum $\mathrm{H}\F_2$ as a $\mathbb{E}_3$-$\mathrm{MU}$-Thom spectrum, the $p$-local Eilenberg-Mac Lane spectra $\mathrm{H}\Z/p^n$ (except for $\mathrm{H}\Z/4$) and $\mathrm{H}\Z_{(p)}$ as $\mathbb{E}_3$-$\mathrm{MU}_{(p)}$-Thom spectra, and the connective complex K-theory spectrum $ku$ as $\mathbb{E}_1$-$\mathrm{MU}$-Thom spectra. We accompany these ``go" results with ``no-go" results that show these ``go" results do not admit $\mathbb{E}_{\infty}$-refinements.
\end{abstract}
\maketitle
\tableofcontents
\input{introduction}

\input{preliminaries}

\input{thom_spectra}

\input{thh}
\bibliographystyle{amsalpha}
\bibliography{ref}
\medskip
\noindent
\textsc{Department of Mathematics, Northwestern University}\\
{\ttfamily prestoncranford2027@u.northwestern.edu}
\end{document}

%% file: introduction.tex
\section{Introduction}
The areas of math $p$-adic arithmetic geomtry and homotopy theory have have experienced rapid development in the last decade. 
This time of rapid development begins with \cite{bhattmorrowscholzeTopologicalHochschildHomology2019} where the Bhatt-Morrow-Szhole study the ($p$-completed) Topological Hochschild Homology $\THH(R)_p^{\wedge}$ of suitable discrete commutative rings known as quasisyntomic rings. 
Furthermore the authors study a filtration on $\THH(R)_p^{\wedge}$ known as the motivic filtration $\fil_{\mot}^{\ast}\THH(R)_p^{\wedge}$. 
Topologiical Hochschild Homology has a property known as being \emph{cyclotomic} that allows Bhatt-Morrow-Szhole to also study variations of the motivic filtration $\fil_{\mot}^{\ast}\THH(R)_p^{\wedge}$ known as (Nygaard-completed) \emph{prismatic cohomology} and \emph{syntomic cohomology}.

Later, in \cite{bhattPrismaticFgauges2022}, \cite{bhattlurieAbsolutePrismaticCohomology2022},  \cite{drinfeldPrismatization2024}, Bhatt-Scholze, Bhatt-Lurie, and Drinfield introduce ``stacky" approaches to the motivic filtration  $\fil_{\mot}^{\ast}\THH(R)_p^{\wedge}$ known as prismatization.
For example, to a $p$-complete discrete quasisyntomic ring $R$, there exists a prismatization stack $(\Spf(R)^{\mathcal{N}})_{s=0}$ known as the arithmetic conjugate stack of $R$ that geometrizes the motivic filtration $\fil_{\mot}^{\ast}\THH(R)_p^{\wedge}$ in the sense that
\[
	\gr_{\mot}^t\THH(R)_p^{\wedge} \simeq R\Gamma\left((\Spf(R)^{\mathcal{N}})_{s=0}, \mathcal{O}\{t\}\right)[2t]
\]
as graded commuative rings.

On the side of homotopy theory, in \cite{hahnwilsonRedshiftMultiplicationTruncated2022}, Hahn-Raksit-Wilson study a filtration $\fil_{\E_{\infty}-\ev}^{\ast} \THH(R)$ on Topological Hochschild of suitable commutative ring spectra known as chromatically quasisyntomic ring spectra. Later, in \cite{pstragowskiPerfectEvenModules2023}, Pstragowski studies a filtration $\fil_{\ev}^{\ast} \THH(R)$ on Topological Hochschild of suitable associative ring spectra known as $\E_2$-ring spectra. 

The study ``stacky" approaches to motivic filtrations for ring spectra, which can be reasonably called the study of motivic filtrations for spectral schemes, is under investigation in the math community. In \cite{gregoricEvenPeriodizationSpectral2025}, Gregoric studies a notion of spectral stacks of commutive ring spectra, called even periodization, that recovers the motivic filtration $\fil_{\E_{\infty}-\ev}^{\ast} \THH(R)$ of commutative ring spectra $R$. In his thesis \cite[Chapters 2 and 7]{devalapurkarSpherochromatismRepresentationTheory2025}, Devalapurkar describes ongoing work (with collaborators) on the following topics: the study even stacks of commutative ring spectra, a ``stacky" version of the even filtration of commutative ring spectra, the study of ``decompleted cycltomic spectra" (in the Nygaard-completion sense), and the study of spectral prismatization stacks of commutative ring spectra. One such example of a spectral prismatization stacks of commutative ring spectra includes a proposed definition of the conjugate stack associated to a suitable commutative ring spectrm.

In this article we propose a definition of conjugate stack $R^{\conj}$ associated to certain ring spectra of a class of (mildly commutative) associative ring spectra that includes the truncated Brown-Peterson spectra $\BP\langle n \rangle$ at all heights $n$. We list three points that defend our proposed definition.
\begin{enumerate}
	\item 
		Recall that given $p$-complete discrete quasisyntomic ring $R$, there exists its associated conjugate $(\Spf(R)^{\mathcal{N}})_{s=0}$ that
		geometrizes the motivic filtration $\fil_{\mot}^{\ast}\THH(R)_p^{\wedge}$ in the sense that there exists an equivalence
		\[
			\gr_{\mot}^t\THH(R)_p^{\wedge} \cong R\Gamma\left((\Spf(R)^{\mathcal{N}})_{s=0}, \mathcal{O}\{t\}\right)[2t]
		\]
		as graded commutative rings. Likewise, for $R$ a suitable ring spectrum so that $R^{\conj}$ is defined, $R^{\conj}$ geomtetrizes the even filtration $\fil_{\ev}^{\ast} R$ in the sense that there exists an equivalence
		\[
		\gr_{\ev}^t\THH(R) \simeq R\Gamma\left(R^{\conj}, \mathcal{O}\{t\}\right)[2t]
		\]
		as spectra for each $t$.
	\item
		For $R$ a $p$-complete discrete quasisyntomic ring $R$ that is suitable as a ring spectrum so that both $(\Spf(R)^{\mathcal{N}})_{s=0}$ and $R^{\conj}$ are defined, the two stacks agree. In paticular, there exists an equivalence
		\[
			\widehat{R^{\conj}}_p \cong (\Spf(R)^{\mathcal{N}})_{s=0}
		\]
		between formal $p$-competion $\widehat{R^{\conj}}_p$ of $R^{\conj}$ and the arithmetic conjugate stack $(\Spf(R)^{\mathcal{M}})_{s=0}$.
	\item
		For $R$ a stuiable commutative ring spectrum so that $R^{\conj}$ and Devalapurkar's proposed definition in \cite[Chapter 7]{devalapurkarSpherochromatismRepresentationTheory2025} (studied in Devalapurkar-Hahn-Raksit-Yuan's work  \cite{devalapurkarPrismatizationCyclotomicSpectra2025} in preparation) of conjugate stack associated to $R$ are defined, the two definitions coincide. Indeed, whereas Devalapurkar defines an even stack $\Spev(R)$ associated to $R$ and defines the conjugate stack in terms of even stacks, one gets that Devalapurkar's definition gives that the conjugate stack associated to $R$ is $\Spev(\THH(R))$ itself, and one can show that $R^{\conj}\simeq\Spev(\THH(R))$ as graded stacks.
\end{enumerate}

We introduce our definition of conjugate stack as well as our Main Theorem B.

\begin{definition}
	Let $R$ be a connective $\E_3$-$\MU$-algebra such that $\THH(R/\MU)$ is even, and define conjugate stack $R^{\conj}$ of $R$ denote the graded quotient stack associated to the graded groupoid scheme
	\[
		\begin{tikzcd}
			\Spec(\THH_{2\ast}(R/\MU)) & \Spec(\THH_{2\ast}(R/\MU^{\otimes 2})).
			\arrow["{\eta_L}"', shift right, from=1-2, to=1-1]
			\arrow["{\eta_R}", shift left, from=1-2, to=1-1]
		\end{tikzcd}
	\]
	where $\eta_L$ and $\eta_R$ denote the maps induced by the left and right units of the Hopf algebroid $(\MU, \MU \otimes \MU)$.
\end{definition}
\begin{notation}
	We establish notation to state Theorem B. We let $\W$ denote the $p$-typical Witt vector group scheme, $F$ denote Frobenius, and $W[F^n]$ denote the kernel of $F^n$ group scheme.
\end{notation}
\begin{namedthm*}{Theorem B}\label{thm:B}
	\begin{itemize}
		\item[(i)] Let 	Let $R$ be a connective $\E_3$-$\MU$-algebra such that $\THH(R/\MU)$ is even. Then $R^{\conj}$ satisdies the the following properties. Firstly, there is an equivalence of spectra
	\[
		\gr_{\ev}^t\THH(R) \simeq R\Gamma\left(R^{\conj}, \mathcal{O}\{t\}\right)[2t]
	\]
	for $\fil_{\ev}^{\ast}$ denoting Pstragowksi's even filttration \cite[Theorem 1.3]{pstragowskiPerfectEvenModules2023}. Furthermore is a strongly convergent multiplicative spectral sequence of signature
	\[
		H^s(R^{\conj}, \mathcal{O}\{t\})\Rightarrow \pi_{2t-s}\THH(R).
	\]
	Lasty, if $R$ is furthermore a discrete $p$-complete quasisyntomic commutative ring such that $\THH(R/MU)$ has bounded $p$-power torsion then there is an isomorphism of fqpc stacks
	\[
		\widehat{R^{\conj}}_p \cong (\Spf(R)^{\mathcal{N}})_{s=0}
	\]
	between the formal $p$-competion of $R^{\conj}$ and the zero fiber $(\Spf(R)^{\mathcal{M}})_{s=0}$ of the Nygaard filtered prismatization of $\Spf(R)$, the familiar definition of conjugate stack of $R$.

	\item[(ii)] Let ${\BP\langle n \rangle}^{\conj}_{p=\cdots=v_n=0}$ denote the mod $(p,v_1\dots, v_n)$ fiber
		\[
			\BP\langle n \rangle^{\conj} \times_{\Spec(\Z_p[v_1,\dots, v_n])/\G_m} B\G_{m, \F_p}
		\]
		of $\BP\langle n \rangle^{\conj}$. We have
		\[
		\BP\langle n \rangle^{\conj}_{p=\cdots=v_n)}
			=
			[F^{m+1}_{\ast}\W / F_{\ast}\W]/\G_m
			=
			[\G_a / \W[F^n]]/\G_m
		\]
		as affine group schemes over $\Spec(\Z_p[v_1,\dots, v_n])/\G_m$ where the first  quotient is associated to $pF^n:F_{\ast}\W_{\F_p}\rightarrow F_{\ast}^{n+1}\W_{\F_p}$ and the second quotient is associated to the shift by $w_0$ $\W[F^n]\rightarrow \G_a$. 
	\end{itemize}
\end{namedthm*}

\begin{corollary}[Corollary~\ref{cor:bok_stack}]\label{intro_bok}	
We have
	\[
		[\Spec(\THH_{\ast}(\F_p))/\G_m]
		\cong
		\F_p^{\conj}
		\cong
		\G_a/\G_m
	\]
	as stacks over $\Spec(\F_p)/\G_m$ and that
	\[
		\bigoplus_{t \in \Z} R\Gamma(\F_p^{\conj}, \mathcal{O}\{t\})[2t]
		\cong
		\F_p[x],
	\]
	with $\abs{x}=2$, as graded $\F_p$-modules. Thus the fact that $\F_p^{\conj}\cong \G_a/\G_m$ as stacks over $\Spec(\F_p)/\G_m$ is a ``stacky" version of B\"oksdedt periodicity.
\end{corollary}

\begin{corollary}[Corollary~\ref{cor:BPn}]\label{intro:BPn}
	We have
	\[
		\Spec\left(\THH_{\ast}(\BP\langle n \rangle)/(p, \dots, v_n)\right)
		\cong
		\BP\langle n \rangle^{\conj}_{p=v_1=\cdots=v_n=0}
		\cong
		[F^{n+1}_{\ast}W/F_{\ast}W]
	\]
	as graded affine group schemes over $\Spec(\F_p)$ and that
	\[
		\bigoplus_{t \in \Z} R\Gamma(\BP^{\conj}_{p=v_1=\cdots=v_n=0}, \mathcal{O}\{t\})[2t]
		\cong
		\F_p[\mu_{n+1}]\otimes \Lambda_{\F_p}(\lambda_1, \dots, \lambda_{n+1}),
	\]
	with $\abs{\mu_{n+1}}=2p^{n+1}$ and $\abs{\lambda_i}=2p^i-1$, as graded $\F_p$-algebras. This is a ``stacky" version of the calculation
	\[
		\pi_{\ast}\left(\gr_{\mot}^{\ast}\THH(\BP\langle n \rangle)\right)/(p, v_1, \dots, v_n)
		\cong
		\F_p[\mu_{n+1}]\otimes \Lambda_{\F_p}(\lambda_1, \dots, \lambda_{n+1})
	\]
	of \cite{angelini-knollSyntomicCohomologyTruncated2026}.
\end{corollary}

To study conjugate stacks $R^{\conj}$ of suitable ring spectra $R$, we first develop techinques for understanding $\THH_{\ast}(R/\MU)$ as a $R_{\ast}$-Hopf algebra.

Consider B\"okstedt periodicity \cite{bokstedtTopologicalHochschildHomology1985}, the fundamental calculation that $\THH_{\ast}(\F_p)\simeq H_{\ast}(\Omega S^3; \F_p)$ as graded (discrete) rings.
The more modern and topological way to prove Bökstedt periodicity requires understanding $\F_p$ as a Thom spectrum. The expression of $\F_p$ as a Thom spectrum is known as the Hopkins-Mahowald Theorem (Theorem \ref{thm:Hopkins-Mahowald}). The $p=2$ case is due to Mahowald in \cite{mahowaldRingSpectraWhich1979}, and the $p>2$ case is due to Mike Hopkins.

Our stratedgy to calculate $\THH_{\ast}(R/\MU)$ (for suitable $R$) as a $R_{\ast}$-Hopf algebra (to prove Theorem B) is summarized following way. Will be to firstly express $Z/p^n$ (but not for $Z/4$) and $\Z_{(p)}$ as multiplicatively coherent generalized Thom spectra in $\Mod_{\MU}$ as stated in the following Theorem A. Secondly, we study a ``$\THH$-of-Thom spectra" formula that accounts not just for algebraic structure but also agumented and coalgebraic structure provided by the geometry of generalized Thom spectra as stated in the following Proposition~\ref{into:THH_bialg}. Lastly we deduce the $\F_p$-Hopf algebra structure on $\THH_{\ast}(\BP\langle n \rangle/\MU)$ by studying it as a $\F_p$-Hopf algebra of $\THH_{\ast}(\F_p/\MU)$.

\begin{notation}
For the following Theorem A we explain the following pieces of notation. We let $\ku$ denote the connective complex K-theory spectrum. We define the infinite loop spaces
\[
	U := \Omega^{\infty}\Sigma \ku, 
	\text{ }
	SU := \Omega^{\infty}\Sigma^3 \ku,
	\text{ and }
	BBSU := \Omega^{\infty}\Sigma^5 \ku
\]
whose infiite loop space structures are determined by the (canonical) commutative ring structure on $\ku$. For a prime $p$ and positive integer $n$ (except for $(p,n)=(2,2)$) we define $\alpha(p,n)$ to be the unique power $p^m$ of $p$ such that the $p$-adic valuation of $v_p((1+p)^{p^m}-1)$ of $(1+p)^{p^m}-1$ is $n$. Lastly we define $U[\alpha(p,n)]$ to be the infite loop space $SU\times_{\varphi(p,n)}S^1$ where encoded the $\alpha(p,n)$-twist on $S^1$. Note that $U[1]=1$.
\end{notation}

\begin{namedthm*}{Theorem A}
	The following statemetns are true.
	\begin{itemize}
		\item 
			We have that $\F_2$ is an $\E_3$-$\MU$-Thom spectrum over $U$ but is not an $\E_4$-$\MU_2^{\wedge}$-Thom spectrum over $U$.
		\item 
			For all primes $p$ and integers $n>0$ other than $(p,n)$, we have that $\Z/p^n$ is an $\E_3$-$\MU_{(p)}$-Thom spectrum over $U[\alpha(p,n)]$ but is not an $\E_{2p+2}$-$\MU_p^{\wedge}$-Thom spectrum over $U[\alpha(p,n)]$. 
			Furthermore $\F_p$ is not an $\E_4$-$\MU_p^{\wedge}$-Thom spectrum over $U$.
		\item 
			We have that $\Z_{(p)}$ is an $\E_3$-$\MU_{(p)}$-Thom spectrum over $Su_p^{\wedge}$ but is not an $\E_{2p+2}$-$\MU_p^{\wedge}$-Thom spectrum over $SU_p^{\wedge}$.
		\item 
			We have that $\Z$ is an $\E_1$-$\MU$-Thom spectrum over $SU$ but is not an $\E_2$-$\MU$-Thom spectrum over $SU$.
		\item 
			We have that $\ku$ is an $\E_1$-$\MU$-Thom spectrum over $BBSU$ but for all primes $p$ that $\ku_p^{\wedge}$ is not an $\E_2$-$\MU_p^{\wedge}$-Thom spectrum over $BBSU_p^{\wedge}$.
	\end{itemize}
\end{namedthm*}

\begin{proposition}[Proposition~\ref{prop:bialg}]\label{intro:bialg}
	Fix $n>1$ and $\mathrm{R}$ an $\E_{n+1}$-ring. Let $f\colon X\rightarrow \Pic(\mathrm{R})$ denote a pointed $\E_n$-space map. Let $T(f)$ denote its corresponding $E_n$-$R$-Thom spectrum and suppose that $T(f)$ admits an $\E_{n+1}$-$R$-algebra structure not necessairly refining its $\E_n$-$R$-algebra structure from being an $\E_n$-$R$-Thom spectrum of $f$Then there exists an equivalence
	\[
		\THH(T(f)/R) \simeq T(f)[BX]
	\]
	of counital $\E_1$-coalgebras in augmented $\E_{n-1}$-$T(f)$-algebras.
\end{proposition}

\subsection*{Conventions}
We say category to mean $\infty$-category=$(\infty,1)$-category
For $X$ a space and $R$ a spectrum we let $R[X]$ denote the homology spectrum $R \otimes \Sigma_+^{\infty}X$. 
We let $\Sp$ denote the category of spectra, let $\Spc$ denote the category of spaces (=anima, $\infty$-groupoids, weak homotopy types), and let $\Spc_{\ast}$ denote the category of pointed spaces.
Spaces and spectra are assumed to neither be localized nor completed unless otherwise notated except for the Brown-Peterson spectra $\BP$ and their truncated variants $\BP\langle n \rangle$ which should be assumed to be $p$-localized.
All spaces are assumed to be pointed and grouplike unless otherwise stated.
All sheaves and stacks are assumed to be sheaves and stacks for the fpqc topology unless otherwise stated.

\subsection*{Acknowledgements}
The author has benefited from conversation on this paper with the following people: Ben Antieau, Sanath Devalapurkar, Jeremy Hahn, Ishan Levy, Deven Manam, Arpon Raksit, Noah Riggenbach, and Allen Yuan.
We thank again thank Ben Antieau for reading a preliminary draft of this paper, Jeremy Hahn for encouraging the author's work on Theorem A, and Devalapurkar for encouraging the author's work on Theorem B.
This work was supported by the National Science Foundation grant DGE-2234667 and Simons Foundation grant SDF-00014770.

\subsection*{Large Language Model Useage Disclosure}
The author used several OpenAI LLM models, including GPT-Sol 6.1, for feedback on writing, both mathematically and expositionally. Some proofs in this paper reflect feedback from the LLM model. In writing this paper, the author originally had a (incorrect) proof of $\Z$ integrally admitis an $\E_3$-$\MU$-Thom spectrum structure, and the proof strategy for proving Theorem 3 part (xi) was suggested out of LLM feedback. Any LLM feedback that has influenced this paper has been only carefully and skeptically considered. The writing in this paper is entirely human generated, and the large scale mathematical ideas of this paper were entirely human generated. The author maintains responsibility for the accuracy of the entirety of this paper.

%% file: preliminaries.tex
\section{Preliminaries}\label{sec:prelim}

In Section~\ref{subsec:thom_spectra} we review generalized Thom spectra. In Subsection~\ref{subssec:coords} we review coordinates of complex orientations of spectra as well as torsors for $\E_n$-monoidal complex orientations. In Section~\ref{subsec_DL} we review and calculate some Dyer-Lashof operations.

\subsection{Generalized Thom Spectra}\label{subsec:thom_spectra}

Our discussion largely follows \cite{andoblumberggepnerhopkinsrezkCategoricalApproachRline2014} on the $\infty$-categorical approach to Thom spectra as well as \cite{antolin-camarenabarthelSimpleUniversalProperty2019} on universal properties of generalized Thom spectra.

For this subsection we fix $R$ an $\E_1$-ring. We begin by defining generalized Thom spectra, also known as $R$-Thom spectra. Let $\Pic(R)$ denote the maximal sub-$\infty$-groupoid of $\otimes_{R}$-invertible objects in $\Modl_R$. 
\begin{definition}
Fix a pointed space map $X \xrightarrow{f}\Pic(R)$. 
The $R$-Thom spectrum $T(f)$ associated to the pointed space map $f$ is the colimit
\[
\colim(X \xrightarrow{f}\Pic(R) \rightarrow \Mod_{R}).
\]
in $\Modl_{R}$.
\end{definition}
The process of associating a $R$-Thom spectrum to a pointed space map into $\Pic(R)$ assembles into a colimit-preserving symmetric monoidal functor from the Cartesian symmetric monoidal overcategory ${\Spc_{\ast}}_{\Pic(R)}$ of pointed spaces over the space $\Pic(R)$ to $\Modl_{R}$. as stated in \cite[Theorem 1.6]{andoblumberggepnerParametrizedSpectraMultiplicative2018}. 
A pointed space map $X\rightarrow \Pic(R)$ where $X$ is connected factors through the connected component $\BGL_1(R)$ of $\Pic(R)$ containing the identity $R$.
\begin{description}
    \item[Multiplicative Coherence]
        The Thom spectrum functor, from $\E_n$-monoidal category overcategory of pointed $\E_n$-spaces over $\Pic(R)$ to $R$-mod, is $\E_n$-monoidal. \cite[Corollary 3.2]{antolin-camarenabarthelSimpleUniversalProperty2019}

\item[Universal Property]
        Let $f:X\rightarrow\Pic(R)$ be an $\E_{n+1}$-monoidal map with $\E_n$-$R$-Thom spectrum $T(f)$. For $A$ an $\E_n$-$R$-algebra, there exists an equivalence of spaces between the space of $\E_n$-monoidal maps from $T(f)$ to $A$ and $\E_n$-null-homotopies of
    \[
    X\xrightarrow{f}\Pic(R)\rightarrow \Pic(\mathrm{A}).
    \]
    \cite[Theorem 3.5]{antolin-camarenabarthelSimpleUniversalProperty2019}This was first proven in \cite{andoblumberggepnerhopkinsrezkUnitsRingSpectra2014} in the $\E_{\infty}$-monoidal setting. See \cite[dsicussion after Theorem 4.3]{hahnyuanExoticMultiplicationsPeriodic2020}.

\item[Thom Isomorphism](\cite{antolin-camarenabarthelSimpleUniversalProperty2019}, Proposition 3.16]
This was first proven for spectra in \cite[Theorem 1.2]{mahowaldRingSpectraWhich1979})
    Let $f\colon X\rightarrow \Pic(R)$ be an $\E_n$-monoidal map. Then there is an equivalence
    \[
    T(f) \otimes_{R} T(f) \simeq T(f) \otimes_R R[X]
    \cong T(f)[X]
\]
    of $\E_n$-$R$-rings.

\item[Base Change]
    For $A\rightarrow B$ an $\E_1$-ring map and $X\xrightarrow{f} \BGL_1(A)$ an $\E_n$-space map then the $\E_n$-$B$-Thom spectrum of the composition $X\rightarrow \BGL_1(A)\rightarrow \BGL_1(B)$ is $T(f) \otimes_{A} B$.
    If $X\rightarrow \BGL_1(A)\rightarrow \BGL_1(B)$ is null then there is an equivalence $T(f) \otimes_A B \simeq B[X]$ of $\E_n$-$B$-algebras.

\item[Intermediate Thom Spectrum Theorem] 
    (\cite{devalapurkarHigherChromaticThom2024}, Proposition 2.1.6)
    Let $X\xrightarrow{i} Y \rightarrow Z$ be a fiber sequence of $\E_n$ spaces with $Z$ connected. Let $f:Y\rightarrow\Pic(R)$ be an $\E_n$-monoidal map. Then there exists an $\E_n$-monoidal map $g: Z\rightarrow \BGL_1(T(f\circ i))$ whose $R$-Thom spectrum $T(g)$ is equivalent to $T(f)$ as $\E_{n-1}$-rings. The following diagram expresses the claims.
\[\begin{tikzcd}
	X & Y & Z \\
	& {\BGL_1(R)} & {\BGL_1(T(f\circ i))}
	\arrow["i", from=1-1, to=1-2]
    \arrow["f \circ i", from=1-1, to=2-2]
	\arrow[from=1-2, to=1-3]
	\arrow["f", from=1-2, to=2-2]
	\arrow["g", from=1-3, to=2-3]
	\arrow[from=2-2, to=2-3]
\end{tikzcd}\]
\end{description}

In this paper we often consider $\E_n$-$\MU$-Thom spectrum.

\begin{example}\label{ex:MU-Thom}
    Consider the mapping space $\Map_{\Alg_{\E_n}(Sp)}(\MU, \MU)$ of $\E_n$-ring maps $\MU\rightarrow \MU$. 
    This mapping space is nonempty as it contains $\id_{\MU}$ which can be regarded as an $\E_n$-ring map for all $n$. For all $n>0$ there is an equivalence of mapping spaces
    \begin{align*}
        \Map_{\Alg_{\E_n}(\Sp)}(\MU, \MU) 
        &\simeq\Map_{\Alg_{\E_n}(\Sp)}(\Sigma^{\infty}BU, \MU)\\
        &\simeq\Map_{\Alg_{\E_n}(\Spc)}(BU, \GL_1(\MU))\\
        &\simeq\Map_{\Alg_{\E_{n-1}}(\Spc)}(SU, \BGL_1(\MU)).
    \end{align*}
    where the first equivalence follows from the Thom isomoprhism associated to the Thom spectrum $\MU$, the second equivalence follows from definition of the spectrum of units $\mathrm{gl}_1(\MU)$, the third equivalence follows from the $(B,\Sigma^{\infty})$ adjunction, and the fourth and last equivalence follows from delooping once. This equivalences of mapping spaces is stated in \cite[beginning of Section 6 through Proposition 6.2]{chadwickmandellEGenera2015}.

    TWe learn that an $E_n$-ring map $\MU\rightarrow \MU$ determines an $\E_{n-1}$-$\MU$-Thom spectrum by considering the associated $\E_{n-1}$-space map $SU\rightarrow\BGL_1(\MU)$ and taking its associated $\E_{n-1}$-$\MU$-Thom spectrum.  
\end{example}

\subsection{Complex Orientations}\label{subssec:coords}
In this subsection we review complex orientations of $\E_n$-ring spectra and their corresponding coordinates. We also review $\E_1$-$R$-Thom spectra of maps $\SU\rightarrow \BGL_1(R)$ for $R$ an even $\E_{\infty}$-ring.

Let $R$ be a homotopy commutative even ring spectrum. Recall that a complex orientation of $R$ is a choice of coordinates of the formal group $\Spf(R^*(\C\P^{\infty}))$ compatible with $\G_m$-action. Concretely, the coordinates of a complex orientation of $R$ is the same data as a power series
$$
x+\sum_i^{\infty} a_{i}x^{i+1}
$$
for $a_i \in \pi_{2i}(R)$ and $x$ denoting the standard complex orientation of $\MU$. The coordinates of $R$ form a $\mathbb{G}_{\inv}(R)$-torsor, where $\mathbb{G}_{\inv}(R)$ denotes the automorphism group of the unit disc, with action $\mathbb{G}_{\inv}(R)$ acting on coordinates by composition.
        
We recall that an equivalent definition of a complex orientation of $R$ is a homotopy commutative ring map $\MU\rightarrow R$. By the classical Thom isomorphism we have an equavelnce
\[\Map_{\text{hoRing}}(\MU, R)\cong \Map_{\text{hoRing}}(\Sigma_+^{\infty}\BU, R)\]
of mapping sets where $\text{hoRing}$ denotes the $1$-category of homotopy commutative ring spectra. The infinite loop space $\BU$ induces a group structure on $\Map_{\text{hoRing}}(\Sigma_+^{\infty}\BU, R)$ that makes $\Map_{\text{hoRing}}(\MU, R)$ into an $\Map_{\text{hoRing}}(\Sigma_+^{\infty}\BU, R)$-torsor.

We remark that the universal property of $\GL_1$ gives a correspondence of mapping sets between $\Map_{\text{hoRing}}(\Sigma_+^{\infty}\BU, R)$ and the set of maps of spaces up to homotopy from $\BU$ to $\GL_1(R)$ The following Proposition \ref{prop:coords} compares the coordinates identifying complex orientation and their corresponding maps of spaces up to homotopy from $\BU$ to $\GL_1(R)$.

\begin{proposition}[\cite{basusagaveschlichtkrullGeneralizedThomSpectra2020}, Section 5.1]\label{prop:coords}
Let $R$ be an even homotopy commutative ring spectrum. A homotopy ring map $\MU\rightarrow R$ with coordinates $a_i \in \pi_{2i}(R)$ determines a map of spaces $\BU\rightarrow \mathrm{GL}_1(R)$ that sends $x_i \in \pi_*(\BU)\cong \Z1$ with $\deg(x_i)=2i$ to $a_{i}\in \pi_{2i}(R)$ in $\mathrm{GL}_1(R)$.
\end{proposition}

We recall facts on multiplicative coherence of complex orientations of even $\E_1$-rings. In \cite{chadwickmandellEGenera2015}, it is determined that complex orientations of even ring spectra always admit $\E_2$-monoidal refinements, yet generally they may or may not admit $\E_4$-monoidal refinements. In \cite[Section 6]{chadwickmandellEGenera2015}, $\E_n$-monoidal complex orientations of $\MU$ are studied by looking at the forgetful maps
\[
\pi_0\Map_{\E_n\text{-Rings}}(\Sigma_+^{\infty}\BU, \MU)\rightarrow \Map_{\text{hoRing}}(\Sigma^{\infty}_+\BU, \MU)
\]
of terms of maps on cohomology
$$H^{*+n}(B^n \BU; \Z)\rightarrow H^*(\BU(1); \Z)$$
where $B^nU$ denotes the $n$-fold classifying space of $U$. In particular, for $n=4$, they study maps $H^{*+4}(BU\langle 6 \rangle; \Z)\rightarrow H^*(\BU(1); \Z)$. For the following we notate $H^*(BU\langle 6 \rangle; \Z)\cong \Z[y_3, y_4, y_5, \dots]$ with $\deg(y_i)=2i$ and $H^*(\BU(1); \Z)\cong \Z[x]$ with $\deg(x)=2$. Their results are stated in Proposition \ref{prop:E4_torsor}.
\begin{proposition}[\cite{chadwickmandellEGenera2015}, Proposition 6.4]\label{prop:E4_torsor}
        The maps $H^{*+4}(BU\langle 6 \rangle; \Z)\rightarrow H^*(\BU(1); \Z)$ are as follows. The generator $y_{m+2}\in H^{2m+4}(BU\langle 6 \rangle; \Z)$ is sent to the following: \[
        \begin{cases*}
        (-1)^{p^t}p^{t-1}x^{p^t-1}&\text{if $m+1=p^t$ for some prime $p$, $t>0$}\\
        (-1)^{m+1}(m+1)x^m&\text{ otherwise }
        \end{cases*}\]
    \end{proposition}
We explain how to use Proposition~\ref{prop:E4_torsor} to study coordinates of $\E_4$-ring automorphisms of $\MU$. Given one $\E_4$-ring map with coordinates $x+a_1x^2+a_2x^3+\dots$, then we can add any $b=p^{t-1}$ or $b=m+1$ (depending on if $i$ is one less than a power of a prime). The result is an $\E_4$-ring map $\MU\rightarrow \MU$ with coordinates
$$
x+a_1x^2+a_2x^3+\dots+{a_{m-2}}x^{m-1}+(a_{m-1}+b)x^{m}+\dots
$$
where the higher coordinates aren't described by the torsor structure. This is because at the $2i$th Postnikov stage $\tau_{\leq 2\ast}B^4GL_1(\MU)$, the choices of lifts to the next stage $\tau_{\leq 2\ast+2}B^4GL_1(\MU)$ for a torsor for $H^{2\ast+2}(BU \langle 6 \rangle), \pi_{2\ast}(\MU)$ with torsor action described by Proposition~\ref{prop:E4_torsor}.

We recall results from \cite{basusagaveschlichtkrullGeneralizedThomSpectra2020} on how $\E_2$-monoidal complex orientations $MU\rightarrow R$ for $R$ even give rise to $\E_1$-$R$-Thom spectra. In particular we recall the following Proposition \ref{prop:E_1-R-Thom} which we've stated in a mildly generalized form.

\begin{proposition}[\cite{basusagaveschlichtkrullGeneralizedThomSpectra2020}, Theorem 1.5]\label{prop:E_1-R-Thom}
For $u_1, \dots,$ a sequence of even-indexed classes in $\pi_{\ast}(R)$, there exists an $\E_1$-monoidal map of spaces $\SU\rightarrow \BGL_1(R)$ with $\E_1$-$R$-Thom spectrum isomorphic to an $\E_1$-$R$-algebra $R/(u_1,\dots,)$.
\end{proposition}

As a consequence of the Thom isomorphism of $\E_n$-algebras and Proposition \ref{prop:E_1-R-Thom}, we have the following Proposition \ref{prop:coords_to_En_MU_thom}.

\begin{proposition}\label{prop:coords_to_En_MU_thom}
	For $(a_1,a_2,\dots)$ the coordinates of an $\E_n$-ring map $\MU\rightarrow\MU$, the corresponding $\E_{n-1}$-space map $\overline{f}:\SU\rightarrow \BGL_1(\MU)$ has $\E_{n-1}$-$\MU$-Thom spectrum isomorphic to $\MU/(a_1,a_2,\dots)$.
\end{proposition}

\subsection{Dyer-Lashof operations}\label{subsec_DL}
In this subsection we review Dyer-Lashof operations and perform some computations of Dyer-Lashof operations. Let $n$ be a fixed positive integer. 

In Lemma \ref{lem:DL_ops_on_MU_algs} we will study Dyer-Lashof operations on some $\F_p$-algebras. Recall that $\F_p$, $\Z$, and $\ku$ each have canonical $\E_{\infty}$-$\MU$-algebra structures. These allow us to consider canonical $\E_{n}$-$\MU$-algebra structures on $R \otimes_{\MU} S$ for $R, S \in \{\F_p, \Z, \ku\}$ by \cite[Section 2.4]{hahnwilsonRedshiftMultiplicationTruncated2022} and using the forgetful functor from $\E_{\infty}$-algebras to $\E_n$-algebras.

\begin{proposition}[\cite{hahnwilsonRedshiftMultiplicationTruncated2022}, Lemma 2.4.1]\label{prop:E1-MU-enveloping-alg}
    The $\E_1$-$\MU$-enveloping algebra $\F_p \otimes_{\MU} \F_p$ of $\F_p$, with its canonical $\E_{2}$-$\F_p$-algebra structure, has homotopy ring
    $$
     \Lambda_{{\F_p}_{\ast}}[\sigma x_0,\sigma x_1, \dots]
    $$
    with $\abs{\sigma x_i}=2i+1$ and Dyer-Lashof operations
    \begin{align*}
        Q^{2^{t+1}}(\sigma x_{2^{t}})&=\sigma x_{2^{t+1}} \text{ mod units}&\text{for all $t\geq 0$, $p=2$}\\
        Q^{t^i}(\sigma x_{2p^t-2})&=\sigma x_{2p^{t+1}-2} \text{ mod units} &\text{for all $t\geq 0$, $p>2$}\\
        Q^{2j+2}(\sigma x_j)&=\sigma x_{2j+1} \text{ mod units decomposables} &\text{for $p=2$}\\
        Q^{j+1}(\sigma x_j)&=\sigma x_{pj+2p-2} \text{ mod units decomposables} &\text{for $p>2$}\\
    \end{align*}
\end{proposition}

\begin{lemma}\label{lem:DL_ops_on_MU_algs}
    Let $p$ be any prime. The sequence of $\E_{\infty}$-$\MU$-$\Z$-algebras
    \[
    \Z \otimes_{\MU} \ku \rightarrow  \Z \otimes_{\MU} \Z \rightarrow  \Z \otimes_{\MU} \F_p,
    \]
    after applying $\pi_*$ becomes
    \[
    \Lambda_{\ku_{\ast}}[\sigma x_1,\sigma x_2, \dots]\rightarrow \Lambda_{\Z_{\ast}}[\sigma x_1,\sigma x_2, \dots] \rightarrow \Lambda_{{\F_p}_{\ast}}[\sigma x_1,\sigma x_2, \dots]
    \]
    with $\abs{\sigma x_i}=2i+1$. Likewise the sequence of $\E_{\infty}$-$\MU$-$\F_p$-algebras
    \[
    \F_p \otimes_{\MU} \ku \rightarrow  \F_p \otimes_{\MU} \Z \rightarrow  \F_p \otimes_{\MU} \F_p,
    \]
    after applying $\pi_{\ast}$ becomes
    \[
    \Lambda_{\ku_{\ast}}[\sigma x_0,\sigma x_1, \dots]\rightarrow \Lambda_{\Z_{\ast}}[\sigma x_0,\sigma x_1, \dots] \rightarrow \Lambda_{{\F_p}_{\ast}}[\sigma x_0,\sigma x_1, \dots]
    \]
    with $\abs{\sigma x_i}=2i+1$. Their $\E_2$ Dyer-Lashof operations are determined mod decomposables (and mod units away from the $\sigma v_i$'s) from the previous proposition. 
\end{lemma}

\begin{proof}
We justify the homotopy ring calculations. By \cite[Proposition 3.6]{angeltveitTopologicalHochschildHomology2008}, and considering how $\ku$ is a quotient of $\MU$ by regular ideal $(x_2, x_3, \dots)$, we have that
$$
\pi_{\ast}(\ku \otimes_{MU}\ku) \cong \Lambda_{\ku_*}[\sigma x_2,\sigma x_3, \dots]
$$
with $\abs{\sigma x_i}=2i+1$. This strategy was used to compute $\pi_*(\F_p \otimes_{\MU} \F_p)$ in the proof of \cite[Lemma 2.4.1]{hahnwilsonRedshiftMultiplicationTruncated2022}. To obtain the other homotopy ring calculations, we recall how $\Z$ is obtained as a quotient of $\ku$ by the regular sequence $(x_2)$ and how $\F_p$ is obtained as a quotient of $\Z$ by the regular sequence $(p)$. 

In all cases, we are considering cofiber sequences associated to a regular quotient $R\rightarrow S$ by $x \in R$. We also have some $A$, a regular quotient of $\MU$ by the ideal $(a_1,a_2,\dots)$ of the form
$$
A \otimes_{MU}R \xrightarrow{1 \otimes x}A\otimes_{MU}R \rightarrow A \otimes_{MU}S,
$$
which allows us to conclude
\begin{align*}
    \pi_*(A \otimes_{MU}S) &= \pi_*(\cofib(A \otimes_{MU}R\xrightarrow{1 \otimes x}\Z \otimes_{MU}R))\\
    &=\pi_*(\cofib(\Lambda_{R_{\ast}}[\sigma a_1, \sigma a_2,\dots] \xrightarrow{x}\Lambda_{R_{\ast}}[\sigma a_1, \sigma a_2,\dots])\\
    &= \Lambda_{S_*}[\sigma a_1, \sigma a_2,\dots].
\end{align*}

We prove the listed Dyer-Lashof operation identities. Consider the sequence of $\E_{\infty}$-$\F_p$-algebras
   $$
    \F_p \otimes_{\MU} \ku \rightarrow  \F_p \otimes_{\MU} \Z \rightarrow  \F_p \otimes_{\MU} \F_p.
    $$
    The result follows from the listed Dyer-Lashof operation identities of Proposition \ref{prop:E1-MU-enveloping-alg} and naturality of Dyer-Lashof operations.
\end{proof}

Now that we've studied the $Q_1$ Dyer-Lashof operations we will need for certain $\E_2$-$\MU$-$\F_p$-algebras, we move on to studying the $Q_1$ Dyer-Lashof operations for certain $\F_p$-homologies of spaces.

\begin{proposition}[\cite{kochmanHomologyClassicalGroups1971}, Theorem 6]
    In $\pi_{ast}(\F_p[\BU])\cong\F_p[b_1,b_2, \dots]$ we have the following identities for the Dyer-Lashof operations:
\begin{align*}
    Q^{2j}(b_n)&=\binom{j-1}{n}b_{n+j} \text{ mod units}&\text{for all $n\geq 1$, $p=2$}\\
    Q^j(b_n)&=\binom{j-1}{n}b_{n+(p-1)j}\text{ mod units} &\text{for all $n\geq 1$, $p>2$}
\end{align*}
\end{proposition}

\begin{proposition}[\cite{kochmanHomologyClassicalGroups1971}]
    In $\pi_{\ast}(\F_p[U])\cong\Lambda_{\F_p}(\sigma b_0, \sigma b_1, \dots)$ and $\pi_{\ast}(\F_p[SU]) \cong \Lambda_{\F_p}(\sigma b_1,\sigma b_2, \dots)$ with $\abs{\sigma b_i}=2i+i$ we have the following identities for the Dyer-Lashof operations mod decomposables:
\begin{align*}
    Q^{2j}(\sigma b_n)&=\binom{j-1}{n}\sigma b_{n+j} \text{ mod units}&\text{for all $n\geq 1$, $p=2$}\\
    Q^j(\sigma b_n)&=\binom{j-1}{n}\sigma b_{n+(p-1)j}\text{ mod units} &\text{for all $n\geq 1$, $p>2$}
\end{align*}
\end{proposition}

\begin{proposition}\label{prop:DL_ops}
 We have the following identities for the $\E_2$ Dyer-Lashof on the polynomial classes $b_i$'s operations mod decomposables:
\begin{align*}
    Q^{2n+2}(\sigma b_n)&=\sigma b_{2n+1} \text{ mod units}&\text{for all $n\geq 1$, $p=2$}\\
    Q^{n+1}(\sigma b_n)&=\sigma b_{pn+(p-1)}\text{ mod units} &\text{for all $n\geq 1$, $p>2$}.
\end{align*}
 We have the following identities for the $\E_4$ Dyer-Lashof on the exterior classes $\sigma b_i$'s operations mod decomposables:
\begin{align*}
    Q^{4}(\sigma b_0)&=\sigma b_{2} \text{ mod units}&\text{for $p=2$}\\
    Q^{2}(\sigma b_0)&=\sigma b_{2p-2}\text{ mod units} &\text{for all $p>2$}.
\end{align*}
We have the following identities for the $\E_{2p+2}$ ($\E_4$ for $p=2$) Dyer-Lashof operations on the exterior classes $\sigma b_i$'s mod decomposables:
\begin{align*}
    Q^{8}(\sigma b_0)&=\sigma b_5 \text{ mod units}&\text{for $p=2$}\\
    Q^{2p}(\sigma b_{p-1})&=\sigma b_{2p^2-p-1}\text{ mod units} &\text{for all $p>2$}.
\end{align*}
\end{proposition}

%% file: thom_spectra.tex
\section{Generalized Thom Spectra and Theorem A}\label{sec:proof_thm_A}
\begin{notation} We make the following definitoins of $\E_{\infty}$-spaces.
	\begin{itemize}
		\item We define
			\[
				U := \Omega^{\infty}\Sigma \ku, 
				\text{ }
				SU := \Omega^{\infty}\Sigma^3 \ku,
				\text{ and }
				BBSU := \Omega^{\infty}\Sigma^5 \ku
			\]
		\item We define the $\E_{\infty}$-ring space $U[n]$ via the pullback square
			\[\begin{tikzcd}
				{U[n]} & U \\
				S^1 & S^1
				\arrow[from=1-1, to=1-2]
				\arrow[from=1-1, to=2-1]
				\arrow[from=1-2, to=2-2]
				\arrow["{[n]}", from=2-1, to=2-2]
			\end{tikzcd}\]
		\item For every prime $p$ and integer $n>0$, other than $(p,n)=(2,2)$, we define $\alpha(p,n)$ to be the unique power $p^m$ of $p$ such that the $p$-adic valuation of $v_p((1+p)^{p^m}-1)$ of $(1+p)^{p^m}-1$ is $n$. Concretely
		 \[
			 \alpha(p,n) :=
			 \begin{cases}
				 p^{n-1} & \text{ if } p>2\\
				 1 & \text{ if } p=2 \text{ and } n=1\\
				 2^{n-2} & \text{ if } p=2 \text{ and } n>2\\
			 \end{cases},
		 \]
		 and we will consider the $\E_{\infty}$-spaces $U[\alpha(p,n)]$.
	\end{itemize}
\end{notation}
\begin{theorem}[Theorem A]\label{thm:A}
	We have the following existence statements:
	\begin{enumerate}
		\item[(i)] 
			There exists an $\E_3$-space map
		 \[
			 U\rightarrow \BGL_1(\MU)
		 \]
		 whose associated $\E_3$-$\MU$-Thom spectrum is isomorphic to $\F_2$.
		 \item[(ii)] 
			 For every prime $p$ and positive integer $n$ (except for $(p,n)=(2,2)$) there exists a $p$-local $\E_3$-space map
		 \[
			 U[\alpha(p,n)]\rightarrow \BGL_1(\MU_{(p)}),
		 \]
		 whose associated $\E_3$-$\MU_{(p)}$-Thom spectrum is isomorhic to $\Z/p^n$.
		\item[(iii)] 
			For every prime $p$ there exists a $p$-local $\E_3$-space map
			\[
				SU \rightarrow \BGL_1(\MU_{(p)})
			\]
		whose associated $\E_3$-$\MU_{(p)}$-Thom spectrum is isomorphic to $\Z_{(p)}$.
    		\item[(iv)] 
			There exists an $\E_1$-space map
			\[
				BBSU\rightarrow \BGL_1(\MU)
			\]
			whose associated $\E_1$-$\MU$-Thom spectrum is isomorphic to $\ku$.
	 \end{enumerate}
     We also have the following non-existance statements:
	 \begin{enumerate}
		\item[(v)] 
			For all primes $p$, there does not exist a $p$-complete $\E_4$-space map
			\[
				U[\alpha(p,1)]\rightarrow \BGL_1(\MU_{p}^{\wedge})
			\]
			whose associated $\E_4$-$\MU_{p}^{\wedge}$-Thom spectrum is isomorphic to $\F_p$.
		\item[(vi)] 
			For all primes $p$ and integers $n>1$ other than $(p,n)=(2,2)$, there does not exist a $p$-complete $\E_{2p+2}$-space map
			\[
				U[\alpha(p,n)]\rightarrow \BGL_1(\MU_{p}^{\wedge})
			\]
			whose associated $\E_{2p+2}$-$\MU_{p}^{\wedge}$-Thom spectrum is isomorphic ot $\Z/p^n$.
		\item[(vii)]
			For all primes $p$ there does not exist a $p$-complete $\E_4$-space map
			\[
				SU\rightarrow \BGL_1(\MU_p^{\wedge})
			\]
			whose associated $\E_{2p+2}$-$\MU_p^{\wedge}$-Thom spectrum is isomorphic to $\Z_p^{\wedge}$.
		\item[(viii)]
			There does not exist a $\E_2$-space map
			\[
				SU \rightarrow \BGL_1(\MU)
			\]
			whose associated $\E_{2}$-$\MU$-Thom spectrum is isomorphic to $\Z$.
		\item[(ix)]
			For all primes $p$ there does not exist a $p$-complete $\E_2$-space map
			\[
				BBSU\rightarrow \BGL_1(\MU_p^{\wedge})
			\]
			whose associated $\E_2$-$\MU_p^{\wedge}$-Thom spectrum is isomorphic to $\ku_p^{\wedge}$.
	\end{enumerate}
 \end{theorem}

 \subsection{Existence Statements for Theorem A}

\begin{lemma}
The $\E_3$-$\MU$-Thom spectrum $T(f)$ of an $\E_3$-space map $f:\SU\rightarrow \BGL_1(\MU)$ orients $\Z$ via its 0th Postnikov truncation map. In other words, the $0$th Postikov truncation map is a map
$$
	T(f)\xrightarrow{\pi_0}{\pi_0}(T(f))\cong \Z
$$
of $\E_3$-$\MU$-Thom spectra.
\end{lemma}
\begin{proof}
	Because $\SU$ is $2$-connected, $\pi_0(T(f))\cong \pi_0(\MU)\cong \Z$ as $\MU$-modules. The fact that the $0$th Postikov truncation map is $\E_3$-monoidal follows from $\pi_0(\--)$ being a symmetric monoidal functor.
\end{proof}
\begin{lemma}\label{lem:E4-MU-to-MU}
	For every prime $p$ there exists a $\E_3$-space map $f_p:SU\rightarrow \BGL_1(\MU)$ with coordinates $(a_1, a_2,\dots)$ satisfying the following two properties:
	\begin{enumerate}
		\item We have $a_{p-1}=x_{p-1}$.
		\item For $i$ such that $p$ doesn't divide $i+1$, the image of $a_i$ mod $p$ and decomposables is nonzero.
	\end{enumerate}
	\begin{proof}
		We construct such an $f_p$ one coordinate at a time. By Proposition~\ref{prop:E4_torsor}, the indicies $i$ for which
		\[
			H^{i+4}(BU\langle 6 \rangle; \F_p) \rightarrow H^i(\BU(1); \F_p)
		\]
		is surjective includes the indicies $i+1=p$ and $i$ such that $p$ doesn't divide $i+1$. Therefore the coordinates $a_i$ for these indicies $i$ may be freely chosen and in paticular may be choosed to satisfy the two properties.
	\end{proof}
\end{lemma}
For the remainder of the section, we fix a prime $p$. Furthermore, we fix $\E_4$-ring map $f:\MU\rightarrow\MU$ as describe in Lemma \ref{lem:E4-MU-to-MU}. We fix its p-localize it so that is an $E_4$-ring map $\MU_{(p)}\rightarrow \MU_{(p)}$ and corresponds to an $\E_3$-space map $\SU_{(p)}\rightarrow \BGL_1(\MU_{(p)})$.
\begin{lemma}\label{lem:p-local-unit-coords}
	The coordinates $a_i$ of $f_p$ satisfy $a_i \equiv u_i x_i$ mod decomposables for $u_i \in (\Z_{(p)})^{\times}$.
\end{lemma}
\begin{proof}
	We are tasked to show that for $i$ such that $p$ divides $i+1$ that the desired claim is true of $a_i$. 
	We will induct on the collection of such $i$ from smallest upwards. 
	The base case of the induction is the case of $i+1=p$ where we have $a_{p-1}=x_{p-1}$.

	Suppose for the sake of induction that for all $j<i$ that $a_i$ satisfied the desired property and $p$ divides $j+1$ so that we seek to show that $a_j$ satisfies the desired property. The image of $a_j$ is also mapepd onto by the $\E_1$-power operation $Q^{2j+2}$ (at the prime $2$) or $Q^{j+1}$ (at odd primes) by some class $a_k$. By inductive hypothesis, $a_k$ satisfies the desired claim, thus $a_j$ satisfies the desired claim by naturality of the $\E_1$-power operation.
\end{proof}
Let $T(p)$ denote the $\E_3-\MU_{(p)}$-THom spectrum associated to $T(f_p)$.
\begin{lemma}\label{lem:E3-Fp-iso}
	The induced $\F_p$-homology map on the $0$th-Postnikov truncation map
	$$
	\F_p \otimes_{\MU} \pi_0(\--):\F_p\otimes T(f) \rightarrow \F_p \otimes_{\MU} \Z
	$$
	is an equivalence of $\E_3$-$\F_p$-$\MU_{(p)}$-algebras.
\end{lemma}
\begin{proof}
	By the Thom isomorphism, there exists an equivalence
	\[
	\F_p[\SU] \rightarrow \F_p\otimes T(f_p) 
	\]
	of $\E_3$-$\F_p$-$\MU_{(p)}$-algebras, therefore we will show that the composition
	$$
	\F_p[SU] \rightarrow \F_p \otimes_{\MU} \Z
	$$
	is an equivalence $\E_3$-$\F_p$-$\MU_{(p)}$-algebras. By the Whitehead Theorem, it is sufficent to show that the associated map of homotopy rings
	$$
	\Lambda_{\F_p}[\sigma c_1, \sigma c_2, \dots] \rightarrow \Lambda_{\F_p}[\sigma x_1, \sigma x_2]
	$$
	is an isomorphism. This follows from Lemma \ref{lem:p-local-unit-coords}.
\end{proof}

A similar proof to that of the above proof of Lemma \ref{lem:E3-Fp-iso} shows that the induces rational homology map on the $0$-th Postnikov truncation map is an equivalence of $\E_3$-$\Q$-$\MU_{(p)}$-algebras.

We now prove Part (iii) of Theorem \ref{thm:A} whose statement we recall now. There exists $\E_3$-space map $f_p:\SU_{(p)}\rightarrow\BGL_1(\MU)$ whose $\E_3$-$\MU$-Thom spectra $T(f)$ has the property that the Postnikov truncation map $T(f)\rightarrow \Z_{(p)}$ is an equivalence of $\E_3$-$\MU_{(p)}$-algebras.

\begin{proof}[Proof of Theorem \ref{thm:A} Part (iii)]
	We will show that the $0$th-Postnikov truncation map $T(f)\rightarrow \Z_{(p)}$ is an equivalence of $\E_3$-$\MU_{(p)}$-algebras. 
	By connectivity and arithmetic fracture square considerations, it is sufficient to show that the 0th-Postnikov truncation map induces isomorphisms of $\E_3$-$\F_p$-$\MU_{(p)}$-algebras upon taking $\F_p$ homology as well as isomorphisms of $\E_3$-$\Q$-$\MU_{(p)}$-algebras upon taking $\Q$ homology. 
	This was studied in Lemma \ref{lem:E3-Fp-iso} for $\F_p$ homology and in subsequent discussion for $\Q$ homology.
\end{proof}
Consider the standard $\E_{\infty}$-orientation $\MU\rightarrow \ku$ so that there is an induced map $\theta:\BGL_1(\MU)\rightarrow \BGL_1(\ku)$.
\begin{lemma}\label{lem:ku_orient}
	A space map $\BU\langle 6 \rangle \rightarrow \fib(B\theta)$ forming the dotted lift in the diagram
	\[\begin{tikzcd}
		& {\fib(B\theta)} & \\
		{BU\langle 6 \rangle} & {B^2GL_1(\MU)} & {B^2GL_1(\ku)}
		\arrow[from=1-2, to=2-2]
		\arrow[dotted, from=2-1, to=1-2]
		\arrow[from=2-1, to=2-2]
		\arrow["{B\theta}", from=2-2, to=2-3]
	\end{tikzcd}\]
	of spaces exists. Furthermore, the choice of space map $\BU\langle 6 \rangle \rightarrow \fib(B\theta)$ determines an $\E_1$-space map $f:\BBSU\rightarrow \BGL_1(\MU)$ that factors (in $\E_1$-spaces) through $\fib(\theta)$ and thus determines an $\E_1$-$\MU$-algebra map $T(f)\rightarrow \ku$.
\end{lemma}

\begin{proof}
	By obstruction theory there exists a space map $\BU\langle 6 \rangle\rightarrow B^2\GL_1(\MU)$ by degree/eveness reasons. We argue that as one builds a map that one can also perform the dotted extension at each Postnikov stage.
	Observe that $\pi_{\ast}(\fib(B\theta))$ is even from analysis of its associated long exact sequence in homotopy rings, and $\pi_{\ast}(B^2\GL_1(\MU))\cong \pi_{\ast-2}(\GL_1(\MU))$, so $\pi_{\ast}(B^2\GL_1(\MU))$ is even. Thus associated obstruction groups
	\[
		H^{2\ast+1}(\BU\langle 6 \rangle; \pi_{\ast}(\fib(B\theta)))
	\]
	to the extension problem of vanish for degree reasons. The resulting claim about orienting $\ku$ follows from the universal property of Thom spectra and that the composition $\BU\langle 6 \rangle \rightarrow B^2\GL_1(\ku)$ factors through $\fib(B\theta)$ and therefore determines the desired $\E_1$-nullhomotopy.
\end{proof}

\begin{lemma}\label{lem:ku_orient_iso}
	There exists $\E_1$-space map $f:\BBSU\rightarrow \BGL_1(\MU)$ whose $\E_1$-$\MU$-Thom spectrum $T(f)$ orients $\ku$ (say via $f:T(f)\rightarrow \ku$) in such a way so that the induced map (afeter the canonical base change $\ku\rightarrow \Z$) on $\Z$-homology
	\[
		\Z\otimes_{\MU} f: \Z[T(f)] \rightarrow \Z \otimes_{\MU} \ku
	\]
	has the property that each class $\sigma x_n$ in the homotopy ring
	\[
		\pi_*(\Z \otimes_{\MU} \ku)\cong \Lambda_{\Z}(\sigma x_2, \sigma x_3,\dots)
	\]
	with $\abs{\sigma x_n}=2n+1$ is surjected upon by the map $\pi_{\ast}(\Z\otimes_{\MU} f)$.
\end{lemma}
\begin{proof}
	It suffices to show that the images may be chosen freely, in other words, that the torsor of lifts that occured in the lifting in the proof of Lemma~\ref{lem:ku_orient_iso} has transitive action. We seek to understand $\fib(B\theta)^2(\BU \langle 6 \rangle)$. The Atiyah-Hirezenburch spetral sequence
	\[
		E_2^{s,t}=H^s(\BU\langle 6 \rangle \pi_t(\fib B\theta)) 
		\Rightarrow 
		(fib(\theta))^{s-t}(\BU\langle 6 \rangle)
	\]
	has $E_2=E_{\infty}$-pages by eveness. This gives a non-canonical isomorphism
	\[
		(fib(\theta))^{q}(\BU\langle 6 \rangle)
		\cong
		\bigoplus_{m>m} H^{m+q}(\BU \langle 6 \rangle) \otimes \pi_m(\fib(B\theta)).
	\]
	and it remains to show that the composition
	\[
		H^{2\ast+2}(BU\langle 6 \rangle \pi_{2\ast}\fib(\theta))
		\rightarrow 
		H^{2\ast+1}(BBSU;  \pi_{2\ast}\fib(\theta))
		\rightarrow
		H^{2\ast+1}(BBSU; \pi_{2\ast+1}(Z\otimes_{\MU} \ku))
	\]
	is surjective. This follwos from the first map being induced by the suspension map $H^{\ast+1}(\BU\langle 6 \rangle)\rightarrow H^{\ast}(BBSU)$ and the second map being induced from the map $\Sigma \fib(\theta)\rightarrow \GL_1(Z\otimes_{\MU} \ku))$.
\end{proof}

We are now prepared to prove Theorem \ref{thm:A} part (iv) whose statement we recall now. 	There exists an $\E_1$-space map $BBSU\rightarrow \BGL_1(\MU)$ whose associated $\E_1$-$\MU$-Thom spectrum is isomorphic to $\ku$.

\begin{proof}[proof of Theorem \ref{thm:A} Part (iv)]
	We will any choice of $\E_1$-space map $f:BBSU\rightarrow\BGL_1(\ku)$ from Lemma \ref{lem:ku_orient_iso} has the property that the corresponding orientation $T(f)\rightarrow \ku$ is an $\E_1$-$\MU$-algebra isomorphism

	By connectivity and Whitehead Theorem considerations, it suffices to show that the induced map
	\[
		\Z[T(f)]
		\rightarrow
		\Z \otimes_{\MU} \ku
	\]
	is an isomorphism on homotopy rings. By the Thom isomorphism says
	\[
		\Z[T(f)] \cong \Z[BBSU]
	\]
	as $\E_1$-$\MU$-algebras, and $\Z[BBSU]$ has its homopy ring abstractly isomorphic to that of $\Z \otimes_{\MU} \ku$. To see the $\Z$-homology map is indeed an isomorphism, it suffices to know it is surjective on homotopy rings which was concluded in Lemma \ref{lem:ku_orient_iso}.
\end{proof}

We move on to proving the $\E_3$-$\MU$-Thom spectrum presentation of $\F_2$ as stated in Theorem \ref{thm:A} part (i) as well as the $\E_3$-$\MU_{(p)}$-Thom spectrum presentations of $\Z/p^n$ (but not $\Z/4$) as staed in Theorem \ref{thm:A} part (ii).

For the following lemma we note the following fact:
\[
\pi_1(\BGL_1(\MU_{(p)}))\cong\pi_0(\GL_1(\MU_{p}))\cong \Z_{(p)}^{\times}.
\]
\begin{lemma}\label{lem:map_from_U}
	There exists $\E_3$-space map $f:SU\rightarrow\BGL_1(\MU_{(p)})$ such that the $p$-localized space map $f_p$ has the property that $T(f_p)\simeq\Z_{(p)}$ as $\E_3-\MU_{(p)}$-algebras. Furthermore, $f$ satisfies the property that it factorized through $U$ in such a way that the inclusion $S^1\hookrightarrow U$ gives the map
	\[
		U(1)\hookrightarrow \U\rightarrow SU_{(p)}\xrightarrow{f_p} \BGL_1(\MU_{(p)})
	\]
	that corresponds to $1+p \in \pi_1(\BGL_1(\MU_{(p)}))\cong \Z_{(p)}^{\times}$.
\end{lemma}

\begin{proof}
To produce the desired factorization is equivalent to producing the dotted map in the diagram
\[\begin{tikzcd}
	{BSU} & \\
	{BU \langle 6 \rangle_{(p)}} & {B^4\GL_1(\MU_{(p)})}
	\arrow[dotted, from=1-1, to=2-2]
	\arrow[hook, from=2-1, to=1-1]
	\arrow["{B^3f_p}", from=2-1, to=2-2]
\end{tikzcd}\]
of spaces. The associated obstruction groups
\[
	H^{2\ast+1}\left(\BSU_{(p)}, \BU\langle 6 \rangle_{(p)}; \pi_{2\ast}(B^4\GL_1(\MU_{(p)}))\right)
\]

vanish for degree reasons, so the factorization exists. The ability to choose the class
\[
	S^1_{(p)}=U(1)_{(p)}\hookrightarrow \U_{(p)}\rightarrow \BGL_1(\MU_{(p)})
\]
	up to homotopy comes from the choice permitted in constructing the first Postnikov stage
\[
	[\BSU, K(\Z_{(p)}^4, 4)]\cong H^4\left(\BSU; \Z_{(p)}^{\times}\right)\cong \Z_{(p)}^{\times}.
\]

Finally, one checks that choosing $1+p$ for the firtst Postnikov stage produces a map $\SU_{(p)}\rightarrow \BGL_1(\MU_{(p)})$ satisfying the criteria of Lemma \ref{lem:E4-MU-to-MU}.
\end{proof}

\begin{corollary}\label{cor:alpha_map}
	The composition $\E_3$-space map composition
	\[
		f_{p,n}:=U[\alpha(p,n)]\rightarrow U \rightarrow \BGL_1(\MU_{(p)}),
	\]
	with $U\rightarrow \BGL_1(\MU_{(p)}$ constructed from Lemma~\ref{lem:map_from_U}, has the property that
	\[
	U(1)_{(p)}\hookrightarrow \U[\alpha(p,n)]\rightarrow \BGL_1(\MU_{(p)})
	\]
	corresponds to $(1+p)^{\alpha} \in \pi_1(\BGL_1(\MU_{(p)}))\cong \Z_{(p)}^{\times}$.
\end{corollary}

We prove Theorem \ref{thm:A} part (ii) whose stanament we recall. For every prime $p$ and positive integer $n$ (except for $(p,n)=(2,2)$) the $p$-local $\E_3$-space map
\[
	f_{p,n}:U[\alpha(p,n)]\rightarrow \BGL_1(\MU_{(p)}),
\]
has its associated $\E_3$-$\MU_{(p)}$-Thom spectrum $T(f_{p,n})$ is isomorhic to $\Z/p^n$.
\begin{proof}[Proof of Theorem \ref{thm:A} part (ii)]
	Consider the (split if and only if $\alpha(p,n)=1$) fiber sequence
	\[
		\SU_{(p)} \rightarrow \U[\alpha(p,n)] \rightarrow \mathrm{S}^1_{(p)}
	\]
	of spaces. 
	A triple looping and obstruction theory arguement shows that this fiber sequence is one of $\E_3$-spaces. 
	Since $\mathrm{S}^1_{(p)}$ is connected and the map $f_p:SU_{(p)}\rightarrow \BGL_1(\MU_{(p)})$ has associated $\E_3$-$\MU_{(p)}$-Thom spectrum isomorphic to $\Z_{(p)}$, by the intermediate Thom spectrum theorem, we have the following commutative diagram
    $$\begin{tikzcd}
	    SU_{(p)} & \U[\alpha(p,n)] & {S^1} \\
		      & {\BGL_1(\MU_{(p)})} & {\BGL_1(T(\SU_{(p)}))\rightarrow\BGL_1(\MU_{(p)})}:
	\arrow["i", from=1-1, to=1-2]
	\arrow["f_p",from=1-1, to=2-2]
	\arrow[from=1-2, to=1-3]
	\arrow["f_{p,n}", from=1-2, to=2-2]
	\arrow[from=1-3, to=2-3]
	\arrow[from=2-2, to=2-3]
    \end{tikzcd}$$
	of $\E_3$-spaces. Note that $\BGL_1(T(SU_{(p)}\rightarrow\BGL_1(\MU_{(p)}))=\BGL_1(\Z)_{(p)}$ so the square in the diagram can be rewritten in the following way:
\[\begin{tikzcd}
	U_{(p)} & {S^1_{(p)}} \\
	{\BGL_1(\MU_{(p)})} & {\BGL_1(\Z_{(p)})}
	\arrow[from=1-1, to=1-2]
	\arrow["f_{p,n}", from=1-1, to=2-1]
	\arrow[from=1-2, to=2-2]
	\arrow[from=2-1, to=2-2]
    \end{tikzcd}\]
BWe have $T(f)$ is  isomorphic as an $\MU$-module to the colimit of the diagram
\[
	S^1 \rightarrow \BGL_1(\Z_{(p)}).
\]
which we've calculated to be $(1+p)^{\alpha(p,n)}$ in Lemma \ref{lem:alpha_map}. Since $S^1_{(p)}=B\Z_{(p)}$, $T(f)$ is isomorphic to the coinvariants of $\Z_{(p)}$ acting $\Z_{(p)}$ where the action of $1$ is by $(1+p)^{\alpha(p,n)}$. Since the $p$-adic valuation of $1-(1+p)^{\alpha(p,n)}$ is $n$, we conclude that the $\E_3$-$\MU_{(p)}$-Thom spectrum of $f_{p,n}$ isomorphic to $\Z/p^n$.
\end{proof}

Finally we see how to decude Theorem \ref{thm:A} part (i) whose statement we recall. There an $\E_3$-space map
\[
	\overline{f_2}: U\rightarrow \BGL_1(\MU),
\]
an integral (existing before $2$-localization) analogue of $f_2$, whose associated $\E_3$-$\MU$-Thom spectrum is isomorphic to $\F_2$.

\begin{proof}[Proof of Theorem \ref{thm:A} Part (i)]
	The proof of Lemma \ref{lem:map_from_U} can be slightly modified to show that there exists an $\E_3$-space map $\overline{f_2}:U\rightarrow \BGL_1(\MU_{(p)})$ such that
	\[
		U(1) \hookrightarrow U \xrightarrow{f_2} \BGL_1(\MU_{(p)})
	\]
	corresonds to the (unique) non-trivial element of $\pi_1(\BGL_1(\MU_{(p)}))\cong \Z^{\times} \cong \F_2$. Furthermore the proof of Theorem \ref{thm:A} part (ii) can be slightly modified to conclude Theorem \ref{thm:A} Part (i).
\end{proof}

\subsection{Nonexistance Results of Theorem A}

In this subsection we prove parts (v) through (ix) of Theorem \ref{thm:A}. We begin with parts (v), (vi), and (vii) whose proofs are obervations of non-compatility with power operations. We make frequent use of and Lemma~\ref{lem:DL_ops_on_MU_algs} and Proposition~\ref{prop:DL_ops} without reference.

We prove Theorem \ref{thm:A} part (v) whose statement we recall. For all primes $p$, there does not exist a $p$-complete $\E_4$-space map
\[
	U[\alpha(p,1)]\rightarrow \BGL_1(\MU_{p}^{\wedge})
\]
whose associated $\E_4$-$\MU_{p}^{\wedge}$-Thom spectrum is isomorphic to $\F_p$.

\begin{proof}[Proof of Theorem \ref{thm:A} Part(v)]
	It suffices to show there there does not exist an equivalence
	\[
		g:\F_p[U_p^{\wedge}]\rightarrow \F_p \otimes_{\MU_p^{\wedge}} \F_p
	\]
	of $\E_4$-$\F_p$-algebras as the existance of a $p$-complete $\E_4$-space map
\[
	U[\alpha(p,1)]\rightarrow \BGL_1(\MU_{p}^{\wedge})
\]
with associated $\E_4$-$\MU_{p}^{\wedge}$-Thom spectrum is isomorphic to $\F_p$ would imply th existance of the equivalence $g$ by a Thom isomoprhism arugment.

Note that $g$ induces a homotopy ring map
\[
	\pi_*(g)
	\colon 
	\Lambda_{\F_p}(\sigma c_0, \sigma c_1,\dots)
	\rightarrow
	\Lambda_{\F_p}(\sigma x_0, \sigma x_1,\dots).
\]
with $|\sigma c_i|=|\sigma x_i|$ for all $i\geq 0$. 
Consider $\sigma c_0 =\sigma p \in \pi_1(\F_p[U_p^{\wedge}])$. 
By cite, for $p=2$, the $\E_4$-operation $Q^4$ applied to $\sigma c_0$ is $\sigma c_2 \in \pi_5(\F_2[U_2^{\wedge}])$ mod units and decomposables. 
Likewise, for $p$ odd, the $\E_4$ operation $Q^2$ applied to $\sigma c_0$ is $\sigma c_{2p-2} \in \pi_{4p-3}(\F_2[U_2^{\wedge}])$ mod units and decomposables.. 
Regardless of the parity of the prime $p$, the $\E_4$-Dyer-Lahof operation on $\sigma c_0$ is non-zero. 
However the analogous $\E_4$-operation on $\sigma x_0 \in \pi_1(\F_p \otimes_{\MU_p^{\wedge}} \F_p)$ vanishes.

\end{proof}

We prove Theorem \ref{thm:A} part (vi) whose statement we recall. For all primes $p$ and integers $n>1$ other than $(p,n)=(2,2)$, there does not exist a $p$-complete $\E_{2p+2}$-space map
\[
	U[\alpha(p,n)]\rightarrow \BGL_1(\MU_{p}^{\wedge})
\]
whose associated $\E_{2p+2}$-$\MU_{p}^{\wedge}$-Thom spectrum is isomorphic ot $\Z/p^n$.

\begin{proof}[Proof of Theorem \ref{thm:A} Part(vi)]
	It suffices to show there there does not exist an equivalence
	\[
		g:\Z/p^n[U_p^{\wedge}]\rightarrow \F_p \otimes_{\MU_p^{\wedge}} \Z/p^n
	\]
	of $\E_{2p+2}$-$\F_p$-algebras as the existance of a $p$-complete $\E_{2p+2}$-space map
\[
	U[\alpha(p,n)]\rightarrow \BGL_1(\MU_{p}^{\wedge})
\]
with associated $\E_4$-$\MU_{p}^{\wedge}$-Thom spectrum is isomorphic to $\Z/p^n$ would imply the existance of the equivalence $g$ by a Thom isomoprhism arugment.

Note that $g$ induces a homotopy ring map
\[
	\pi_*(g)
	\colon 
	\Lambda_{\F_p}(\sigma c_0, \sigma c_1,\dots)
	\rightarrow
	\Lambda_{\F_p}(\sigma x_0, \sigma x_1,\dots).
\]
with $|\sigma c_i|=|\sigma x_i|$ for all $i\geq 0$. 
Consider $\sigma c_{p-1}\in \pi_{2p-1}(\F_p[U_p^{\wedge}])$. 
By cite, for $p=2$, the $\E_4$-operation $Q^8$ applied to $\sigma c_{1}$ is $\sigma c_5 \in \pi_{11}(\F_2[U_2^{\wedge}])$. 
Likewise, for $p$ odd, the $\E_{2p+2}$ operation $Q^{2p}$ applied to $\sigma c_{p-1}$ is $-\sigma c_{2p^2-p-2} \in \pi_{4p^2-2p-1}(\F_2[U_2^{\wedge}])$. 
Regardless of the parity of the prime $p$, the $\E_{2p+2}$-Dyer-Lahof operation on $\sigma c_{p-1}$ is non-zero. 
However, the analogous $\E_{2p+2}$-operation on $\sigma x_{p-1} \in \pi_1(\F_p \otimes_{\MU_p^{\wedge}} \Z/p^n)$ vanishes.
\end{proof}

We prove Theorem \ref{thm:A} part (vii) whose statement we recall. For all primes $p$ there does not exist a $p$-complete $\E_4$-space map
\[
	SU\rightarrow \BGL_1(\MU_p^{\wedge})
\]
whose associated $\E_{2p+2}$-$\MU_p^{\wedge}$-Thom spectrum is isomorphic to $\Z_p^{\wedge}$.

\begin{proof}[Proof of Theorem \ref{thm:A} Part(vii)]
	This can be proven by a slight variation of the above proof of part (vi) of Theorem \ref{thm:A}.
\end{proof}

Now that we've proven parts (v), (vi), and (vii) of Theorem \ref{thm:A}, we move on to the proofs of parts (viii) and (ix) whose proofs differ from the proofs of parts (v), (vi), and (vii) in that they aren't completely explained by a power operation obstruction.

We prove Theorem \ref{thm:A} part (viii) whose statement we recall. There does not exist a $\E_2$-space map
\[
	SU \rightarrow \BGL_1(\MU)
\]
whose associated $\E_{2}$-$\MU$-Thom spectrum is isomorphic to $\Z$.

\begin{proof}[Proof of Theorem \ref{thm:A} Part(viii)]
	It suffices to show there there does not exist an equivalence
	\[
		g:\Z[\SU]\rightarrow \Z \otimes_{\MU} \Z
	\]
	of $\E_{2}$-$\Z$-algebras as the existance of a $\E_{2}$-space map
	\[
		SU\rightarrow \BGL_1(\MU)
	\]
	with associated $\E_2$-$\MU$-Thom spectrum is isomorphic to $\Z$ would imply the existance of the equivalence $g$ by a Thom isomoprhism arugment. 

	Thus we let $g$ denote any such map of $\E_{2}$-$\Z$-alebras corresponding to an $\E_3$-ring map $\MU$ to $\MU$ specficied by coordinates $(c_1, c_2, \dots)$, and will show $g$ is not an equivalence. Base changing along $\Z\rightarrow \F_5$ yields a map
	\[
		\F_5 \otimes_{\Z} g: \F_5[\SU]\rightarrow \F_5 \otimes_{\MU} \Z
	\]
	of $\E_2$-$\F_5$-algebras. Consider $h$ the compositon
	\[
	h\colon 
	\F_5 \otimes_{\MU} \Z 
	\rightarrow 
	\F_5 \otimes_{\MU} \F_5 
	\rightarrow 
	\F_5 \otimes_{\F_5\otimes \MU} \F_5 
	\cong 
	\F_5[\SU]
	\]
	where the first map is induces by the base change $\Z\rightarrow \F_5$, the second map is induces by the Hurewicz $\S\rightarrow \MU$, and the last equivalence follows from the Thom isomorphism $\F_5[\MU]\cong \F_5[\BU]$.
	Composing $g$ and $h$ give the map
	\[
		h \circ g: \F_5[\SU]\rightarrow \F_5[\SU]
	\]
	of $\E_2$-$\F_5$-algebras.
	We recall (Milnor-Nikov) that the map $\MU\rightarrow \Z\otimes \MU$ induced Hurewicz $S\rightarrow \Z$ gives the homotopy ring map
	\[
		\Z[x_1, x_2, \dots] \rightarrow \Z[b_1, b_2, \dots]
	\]
	with $|x_i|=|b_i|=2$ for all $i>0$ where $x_i$ is sent to $m_ib_i$ mod decomposables where
	\[
		m_i
		=
		\begin{cases}
			q & i+1=q^n \text{ for some prime } q \\
			1 & \text{otherwise}
		\end{cases}
	\]
	so that $x_1$ is mapped to $2b_i$ and $x_9$ is mapped to $b_9$ mod decomposables.
	We also recall from CITE that the $\E_2$-Dyer Lashof operation
	\[
		Q^2: \pi_3(\F_5[\SU])\rightarrow \pi_{19}(\F_5[\SU])
	\]
	satisfied $Q^2(\sigma b_1)=\sigma b_9$ mod units and decomposables. Putting everythign togther gives by naturalize that $h\circ g$ sends $\sigma b_9$ to $2\sigma b_9$ mod units and decomposables which allows us to conclude that the coordinte $c_9$ cannot be the generator of $\pi_{18}(\MU)$ mod units and decomposables. By Lemma \ref{lem:p-local-unit-coords} this allows us to conclude that that $g$ cannot be an equivalence.
\end{proof}

We prove Theorem \ref{thm:A} part (ix) whose statement we recall. For all primes $p$ there does not exist a $p$-complete $\E_2$-space map
\[
	BBSU\rightarrow \BGL_1(\MU_p^{\wedge})
\]
whose associated $\E_2$-$\MU_p^{\wedge}$-Thom spectrum is isomorphic to $\ku_p^{\wedge}$.

\begin{proof}[Proof of Theorem \ref{thm:A} Part(ix)]
	It suffices to show there there does not exist an equivalence
	\[
		g:\ku[\BBSU]\rightarrow \ku_p^{\wedge} \otimes_{\MU_p^{\wedge}} \ku_p^{\wedge}
	\]
	of $\E_{2}$-$\ku_p^{\wedge}$-algebras as the existance of a $p$-complete $\E_{2}$-space map
	\[
		BBSU\rightarrow \BGL_1(\MU_p^{\wedge})
	\]
	with associated $\E_2$-$\MU_p^{\wedge}$-Thom spectrum is isomorphic to $\ku_p^{\wedge}$ would imply the existance of the equivalence $g$ by a Thom isomoprhism arguement.

	It suffices to show there there is no $\E_1$-$\F_p$-algebra equivalence
	\[
		\F_p[\BU \langle 6 \rangle]
		\rightarrow 
		\F_p \bigotimes_{\F_p \otimes_{\MU_p^{\wedge}} \ku_p^{\wedge}} \F_p
	\]
	as this would follow from the existance of the $\E_{2}$-$\ku_p^{\wedge}$-algebra map $g$ by base changing $\ku_p^{\wedge}\rightarrow \F_p$ then applying the two-sided $\F_p$-bar construction.
	By \cite{andohopkinsstricklandEllipticSpectraWitten2001}, the divided power algebra associative multiplication structure on $\pi_{\ast}(\F_p[\BU \langle 6 \rangle])$ has nilpotence, so it suffices to show that the associative multiplication on
	\[
		\pi_{\ast}\left(\F_p \bigotimes_{\F_p \otimes_{\MU_p^{\wedge}} \ku_p^{\wedge}} \F_p\right)
	\]
	does not have nilpotence. Consider the algebra base change $\ku_p^{\wedge}\rightarrow \F_p$ and the induced base change
	\[
		\F_p \otimes_{\MU_p^{\wedge}} \ku_p^{\wedge}
		\rightarrow
		\F_p \otimes_{\MU_p^{\wedge}} \F_p
	\]
	which itself induces the map
	\[
		\F_p \bigotimes_{\F_p \otimes_{\MU_p^{\wedge}} \ku_p^{\wedge}} \F_p
		\rightarrow
		\F_p \bigotimes_{\F_p \otimes_{\MU} \F_p}\F_p
		\cong
		\THH(\F_p/\MU)
	\]
	of $\E_1$-$\F_p$-algebras. The multiplication on $\THH_{\ast}(\F_p/\MU)$ is polynomial (hahn wilson), and a Bar spectral sequence computation allows us to see that the induced map is injective on homotopy ring classes therefore cannot have nilpotence.

\end{proof}

%% file: thh.tex
\section{Topological Hochschild Homology and Theorem B}\label{sec:THH}
We begin by recalling the definition of conjugate stack.
\begin{definition}
	Let $R$ be a connective $\E_3$-$\MU$-algebra such that $\THH(R/\MU)$ is even, and define conjugate stack $R^{\conj}$ of $R$ denote the graded quotient stack associated to the graded groupoid scheme
	\[
		\begin{tikzcd}
			\Spec(\THH_{2\ast}(R/\MU)) & \Spec(\THH_{2\ast}(R/\MU^{\otimes 2})).
			\arrow["{\eta_L}"', shift right, from=1-2, to=1-1]
			\arrow["{\eta_R}", shift left, from=1-2, to=1-1]
		\end{tikzcd}
	\]
	where $\eta_L$ and $\eta_R$ denote the maps induced by the left and right units of the Hopf algebroid $(\MU, \MU \otimes \MU)$.
\end{definition}
\begin{theorem}[Theorem B]\label{thm:B}
	\begin{itemize}
		\item[(i)] Let 	Let $R$ be a connective $\E_3$-$\MU$-algebra such that $\THH(R/\MU)$ is even. Then $R^{\conj}$ satisdies the the following properties. Firstly, there is an equivalence of spectra
	\[
		\gr_{\ev}^t\THH(R) \simeq R\Gamma\left(R^{\conj}, \mathcal{O}\{t\}\right)[2t]
	\]
	for $\fil_{\ev}^{\ast}$ denoting Pstragowksi's even filttration \cite[Theorem 1.3]{pstragowskiPerfectEvenModules2023}. Furthermore is a strongly convergent multiplicative spectral sequence of signature
	\[
		H^s(R^{\conj}, \mathcal{O}\{t\})\Rightarrow \pi_{2t-s}\THH(R).
	\]
	Lasty, if $R$ is furthermore a discrete $p$-complete quasisyntomic commutative ring such that $\THH(R/MU)$ has bounded $p$-power torsion then there is an isomorphism of fqpc stacks
	\[
		\widehat{R^{\conj}}_p \cong (\Spf(R)^{\mathcal{N}})_{s=0}
	\]
	between the formal $p$-competion of $R^{\conj}$ and the zero fiber $(\Spf(R)^{\mathcal{M}})_{s=0}$ of the Nygaard filtered prismatization of $\Spf(R)$, the familiar definition of conjugate stack of $R$.

	\end{itemize}
	We record the following conjugate stack calculations of $R=\Z/p^n$ and $R=\BP\langle n \rangle$.
	\begin{itemize}
	\item[(ii)] Let $p$ a be prime and $n>0$ an integer such that $(p,n)\neq(2,2)$. Then
		\[
		(\Z/p^n)^{\conj}
			=
			\begin{cases}
				\G_{a, \F_p}/\G_m & n=1 \\
				\G^{\sharp}_{a, \Z/p^n} \times [\G_{a, \Z/p^n} / \G^{\sharp}_{a,\Z/p^n}]/\G_m & n>1
			\end{cases}
		\]
		as affine group schemes over $\Spec(\Z/p^n)/\G_m$ where (for $n>1$) the quotient is associated to the divided-power envelope $\G^{\sharp}_{a,\Z/p^n}\rightarrow \G_{a, \Z/p^n}$.
	\item[(iii)] Let $\left(\mathcal{X}_{\BP\langle n \rangle}\right)_{p=\cdots=v_n)}$ denote the mod $(p,v_1\dots, v_n)$ reduction
		\[
			{\BP\langle n \rangle}^{\conj} \times_{\Spec(\Z_p[v_1,\dots, v_n])/\G_m} B\G_{m, \F_p}
		\]
		of $\BP\langle n \rangle^{\conj}$. We have
		\[
		\BP\langle n \rangle^{\conj}_{p=\cdots=v_n=0}
			=
			[F^{m+1}_{\ast}\W_{\Z_p} / F_{\ast}\W_{\Z_p}]/\G_m
			=
			[\G_a / \W[F^n]]\G_m
		\]
		as affine group schemes over $\Spec(\Z_p[v_1,\dots, v_n])/\G_m$ where the quotient is associated to $pF^n:F_{\ast}\W_{\F_p}\rightarrow F_{\ast}^{n+1}\W_{\F_p}$.
	\end{itemize}
\end{theorem}

\subsection{Bialgebraic Structure $\THH$ of Generalized Thom Spectra}\label{subsec:bialg}

We recall the  THH-of-Thom spectra formula in Proposition \ref{prop:THH_of_Thom_spectra}.
\begin{proposition}[\cite{krausenikolausLecturesTopologicalHochschild}, Proposition 4.7]\label{prop:THH_of_Thom_spectra}
    Let $f:X\rightarrow \BGL_1(\mathrm{R})$ be an $\E_n$-map of spaces for $n\geq 2$, and suppose that the $\E_n$-$R$-algebra structure of $T(f)$ extends to an $\E_{n+1}$-$\mathrm{R}$-algebra structure not necessairly induced from a space level structure on $X$. Then there is an equivalence
    \[
    \THH(T(f)/\mathrm{R})\cong T(f)[BX]
	\]
    of $\E_{n-1}$-$\mathrm{R}$-algebras.
\end{proposition}

\begin{corollary}\label{cor:THH_rel_to_MU}
	Using generalized Thom spectra equivalences of Theorem A and Proposition \ref{prop:THH_of_Thom_spectra}, we have the following equivalences
	\begin{enumerate}
		\item
			There is an equivalence
			\[
				\THH(\F_2/\MU)\cong \F_2[\BU]
			\]
			of $\E_2$-$\MU$-algebras.
		\item For each prime $p$ and $n>0$ (other than $(p,n)=(2,2)$) there is an equivalence
			\[
				\THH((\Z/p^n)/\MU)\cong (\Z/p^n)[\BU[\alpha(p,n)]
			\]
			of $\E_2$-$\MU$-algebras.
		\item For each prime $p$, there is an equivalence
			\[
				\THH(\Z_{(p)}/\MU_{(p)})\cong \Z_{(p)}[\BSU]
			\]
			of $\E_2$-$\MU_{(p)}$-algebras
	\end{enumerate}
\end{corollary}

\begin{proof}
	The only thing needing justificaiton is the $Z/p^n$ case which follows from being $p$-torsion and localization being smashing.`
\end{proof}

We will generalize Proposition \ref{prop:THH_of_Thom_spectra} to bialgebraic structures. We begin by introducing $\E_n$-coalgebras following \cite[Section 3]{beardsleyBialgebrasComodulesDescent2023}. and $(i,j)$-bialgebras following \cite[Section 3.1]{beardsleyBialgebrasComodulesDescent2023} The $\infty$-category of $\E_n$-coalgebras $\CoAlg_{\E_n}(\mathcal{C})$ in an $\E_{n}$-monoidal $\infty$-category $\mathcal{C}$ is defined to be $(\Alg_{\E_n}(\mathcal{C}^{\op}))^{\op}$.

\begin{lemma}[Proposition 4.9 and Corollary 4.14 in \cite{beardsleyBialgebrasComodulesDescent2023}]\label{lem:constant_thom_coalg}
	Fix $n>0$. For $\mathrm{R}$ an $\E_{n+1}$-ring and $X$ an $\E_{m}$-space, let $f:X\rightarrow \Pic(\mathrm{R})$ denote the constant $\E_n$-space map. Then the $R$-Thom spectrum $T(f)\cong R[X]$ is a counital $\E_n$-coalgebra in $\Mod_{\mathrm{R}}$ with coproduct
	\[
	T(f)
	\cong
	R[X]
	\xrightarrow{R[\Delta_X]}
	R[X \times X]
	\cong
	R[X] \otimes_R R[X]
	\cong
	T(f) \otimes_R T(f)
	\]
	and counit
	\[
	T(f)
	\cong
	R[X]
	\xrightarrow{R[X\rightarrow \ast]}
	R.
	\]

	Furthermore $T(f)$ may be regarded as an $(i,j)$-bialgebra in $\Modl_{\mathrm{R}}$ for all $i+j\leq \min(n,m)$ with $i,j\geq 0$. 
	Finally, $T(f)$ may be regarded as a couniatal (right) $R[X]$-comodule in $\Mod_R$ which coaction $\Delta_{T(f)}$ given by the Thom diagonal:
	\[
		T(f)
		\cong
		R[X]
		\xrightarrow{R[\Delta_X]}
		R[X \times X]
		\cong
		R[X] \otimes_{R} R[X]
		\cong
		T(f) \otimes_{\mathrm{R}} T(f).
	\]
\end{lemma}
Recall that there exists a category of augmented of augmented $\E_n$-$R$-modules, that is, the overcategory of $\E_n$-$R$-modules over $R$.
\begin{example}\label{ex:augmentations}
	Fix $n>0$. For $\mathrm{R}$ an $\E_{n+1}$-ring and $X$ an $\E_{n}$-space, let $f:X\rightarrow \Pic(\mathrm{R})$ denote an $\E_n$-space map, and let $T(f)$ denote its corresponding $\E_n$-$R$-Thom spectrum.
	\begin{enumerate}
		\item The $\E_n$-$R$-algebra $T(f)[X]$ admits an $\E_n$-$T(f)$-algebra strcture and admits an augmented $\E_n$-$T(f)$-algebra structure via the map
			\[
				T(f)[X]
				\xrightarrow{T(f)[\ast_X]}
				T(f).
			\]
		which is the counit o
		\item The $\E_n$-$R$-algebra $T(f) \otimes_R T(f)$ admits an $\E_n$-$T(f)$-algebra strcture and further admits an augmented $\E_n$-$T(f)$-algebra strcture via the multiplication
			\[
				\mu_{T(f)}:T(f) \otimes_R T(f)
				\rightarrow
				T(f)
			\]
		on $T(f)$.
	\end{enumerate}
\end{example}
In the following Lemma, we refer to the augmentation structures studied in Examples \ref{ex:augmentations}.
\begin{lemma}\label{lem:augmented_thom_iso}
	Fix $n>0$. For $\mathrm{R}$ an $\E_{n+1}$-ring and $X$ an $\E_{n}$-space, let $f:X\rightarrow \Pic(\mathrm{R})$ denote an $\E_n$-space map, and let $T(f)$ denote its corresponding $\E_n$-$R$-Thom spectrum. Then the Thom isomorphism
	\[
		\phi_f: T(f) \otimes_R T(f) \cong T(f)[X]
	\]
	respects the $T(f)$-augmentations therefore is an isomorphism of augmented $\E_n$-$T(f)$-algebras.
\end{lemma}
\begin{proof}
	This follows from the coherent factorization of the Thom diagonal \cite{beardsleyBialgebrasComodulesDescent2023} and counitality as in Lemma \ref{lem:constant_thom_coalg}.
\end{proof}
Now we state the bialgebraic ``$\THH$ of Thom spectra formula".
\begin{proposition}[Proposition~\ref{intro_bialg}]\label{prop:bialg}
	Fix $n>1$ and $\mathrm{R}$ an $\E_{n+1}$-ring. Let $f\colon X\rightarrow \Pic(\mathrm{R})$ denote a pointed $\E_n$-space map. Let $T(f)$ denote its corresponding $E_n$-$R$-Thom spectrum and suppose that $T(f)$ admits an $\E_{n+1}$-$R$-algebra structure not necessairly refining its $\E_n$-$R$-algebra structure from being an $\E_n$-$R$-Thom spectrum of $f$Then there exists an equivalence
	\[
		\THH(T(f)/R) \simeq T(f)[BX]
	\]
	of counital $\E_1$-coalgebras in augmented $\E_{n-1}$-$T(f)$-algebras.
\end{proposition}

\begin{proof}
	From Lemma~\ref{lem:augmented_thom_iso}, we recall that the Thom isomorphism gives an equivalence
	\[
		\phi_f: T(f) \otimes_R T(f) \simeq T(f)[X]
	\]
	of augmented $\E_n$-$T(f)$-algebras. It suffices to argue the bar construction can be viewed as a functor from augmented $\E_n$-$T(f)$-algebras to $\E_1$-coalgebras in augmented $\E_{n-1}$-$T(f)$-algebras so that applying the bar construction yields the isomoprhism

	\begin{align*}
		B(T(f),\phi_f,T(f)): 
		\THH(T(f)/R)
		&\cong
		T(f)_{T(f) \otimes_R T(f)} T(f)\\
		&\rightarrow
		T(f) \otimes_{T(f)[X]} T(f)\\
		&\cong
		T(f)[BX].
	\end{align*}
	Indeed, by Dunn additivity we have that
	\[
		\Alg_{\E_n}^{\aug}(\Mod_R)=\Alg_{E_1}(\Alg_{\E_{n-1}}^{\aug}(\Mod_R)),
	\]
	and the bar construction can be thought of as a functor
	\[
		\mathrm{Bar}:\Alg_{\E_1}(\mathcal{C})\rightarrow \CoAlg_{\E_1}(\mathcal{C})
	\]
	for $\mathcal{C}=\Alg_{\E_{n-1}}^{\aug}(\Mod_R)$.
\end{proof}

\begin{proposition}\label{prop:thh_of_thom_hopf}
	Fix $n>2$, $\E_{n+1}$-ring $R$, and an $\E_n$-map $f:X\rightarrow \Pic(\mathrm{R})$. Let $T(f)$ denote the $\E_n$-$\mathrm{R}$-Thom spectrum of $f$. Suppose $T(f)_{\ast}[BX]$ is flat over $\pi_{\ast}(T(f))$. Then there exists an equivalence
	\[
		\THH_{\ast}(T(f)/R)\rightarrow T(f)_{\ast}[BX]
	\]
	of graded bicommutative $T(f)_{\ast}$-Hopf algebras.
\end{proposition}

\begin{proof}
	By Lemma \ref{lem:THH_of_thom_bialg}, we have an equivalence 
	\[
		\THH(T(f)/R) \cong R[BX]
	\]
	of counital $\E_1$-coalgebras in augmented $\E_{n-1}$-$T(f)$-algebras. On homotopy rings this induces an equivalence of graded commutative $T(f)_{\ast}$-algebras and graded cocommutative $T(f)_{\ast}$-coalgebras. By flatness of of $T(f)_{\ast}[BX]$ over $\pi_{\ast}(T(f))$, we obtain the desired graded bicommutative $T(f)_{\ast}$-bialgebra equivalence. Finally this automatically upgrades to an equivalence graded bicommutative $T(f)_{\ast}$-Hopf algebras with antipode give by inversion on the group structure of $BX$.
\end{proof}
Finally we state the following calculation of $\THH_{\ast}(\F_p/\MU)$ from Corollary \ref{cor:THH_rel_to_MU} in Corollary \ref{cor:bialg_thh_F_p_MU}. To do so, we recall from \cite[Proposition 4.1]{andohopkinsstricklandEllipticSpectraWitten2001} the equivalence of $\Spec(H_{\ast}(\BU, \F_p))$ and $\mathrm{C}^1(\widehat{\G_a}, \G_m)$ as graded group schemes over $\F_p$. This induces an equivalence of $H_{\ast}(\BU, \F_p)$ with the coordinate ring $\mathcal{O}_{\mathrm{C}^1(\widehat{\G_a}, \G_m)}$ of graded $\F_p$-bialgebras.

\begin{proposition}
	We have
	\[
		\THH_{\ast}(\F_p/\MU)\cong \pi_{\ast}(\F_p[\BU])
	\]
	as bicommutative $\F_p$-Hopf algebras,
	\[
		\THH_{\ast}(\Z/p^n/\MU)\cong \pi_{\ast}\left((\Z/p^n
		)[B(U[\alpha(p,n))]\right)
	\]
	as bicommutative $\Z/p^n$-Hopf algebras and
	\[
		\THH_{\ast}(\Z_{(p)}/\MU)\cong \pi_{\ast}(\Z_{(p)}[\BSU])
	\]
	as bicommutative $\Z_{(p)}$-Hopf algebras.
\end{proposition}

\subsection{Conjugate Stacks}

We now prove Theorem B part (i) whose statement we recall.  Let $R$ be a connective $\E_3$-$\MU$-algebra such that $\THH(R/\MU)$ is even. Thus there exists a conjugate stack $R^{\conj}$. Then $R^{\conj}$ satisfies the following properties. For one,
Then
	\[
		\gr_{\ev}^t\THH(R) \cong R\Gamma\left(R^{\conj}, \mathcal{O}\{t\}\right)[2t].
	\]
		Furthermore there is a strongly convergent spectral sequence of signature
	\[
		H^s(R^{\conj}, \mathcal{O}\{t\})\Rightarrow \pi_{2t-s}\THH(R).
	\]
	Lastly, if $R$ is furthermore a discrete $p$-complete quasisyntomic commutative ring such that $\THH(R/MU)$ has bounded $p$-power torsion then there is an isomorphism of fqpc stacks
	\[
		\widehat{R^{\conj}}_p \cong (\Spf(R)^{\mathcal{N}})_{s=0}
	\]
	between the formal $p$-competion of $R^{\conj}$ and the zero fiber $(\Spf(R)^{\mathcal{M}})_{s=0}$ of the Nygaard filtered prismatization of $\Spf(R)$, the familiar definition of conjugate stack of $R$.

\begin{proof}[Proof of Theorem \ref{thm:B}, Part (i)]
	We first show that
	\[
		\gr_{\ev}^t\THH(R) \cong R\Gamma\left(R^{\conj}, \mathcal{O}\{t\}\right)[2t].
	\]
	Since the augmenteation $\THH(\MU)\rightarrow \MU$ is faithfully evenly flat, base change shows that $\THH(R)\rightarrow \THH(R/\MU)$ is faithfully evenly flat. By faithfully flat descent the natural comparison
	\[
		\fil_{\ev}(\THH(R))\rightarrow
		\lim_{n \in \Delta}
		\fil_{\ev}(\THH(R/\MU^{\otimes n+1}))
	\]
	of filtered spectra is an equivalence asfter Postnikov completion. By the assumption of connectivity of $R$ we learn this natural comparison of filtered spectra is an equivalence before Postnikov completion. For each $n$ we have that $\THH(R/\MU^{\otimes n+1}))$ is even (by assumption that $\THH(R/\MU)$ is even) thus
	\[
		\fil_{\ev}(\THH(R/\MU^{\otimes n+1}))
		\simeq
		\tau_{\geq 2\ast}(\THH(R/\MU^{\otimes n+1}))
	\]
	as filtered spectra. On the other hand by fpqc descent we have
	\[
		R\Gamma\left(R^{\conj}, \mathcal{O}\{t\}\right)
		\simeq
		\lim_{n \in \Delta}
		H(\pi_{2t}\left(\THH(R/MU) \otimes_{\THH(R)} \THH(R/MU)\right)^{\otimes_{\THH(R/\MU)}n}),
	\]
	as (derived) $\Z$-modules. So the arguement finishes by observing the equivalence
	\[
		\THH(R/\MU^{\otimes n+1})
		\simeq
		\THH(R/MU) \otimes_{\THH(R)} \THH(R/MU)^{\otimes_{\THH(R/\MU)}n}
	\]
	as $\E_2$-rings.

	The existance of the strongly convergent spectral sequence of signature
	\[
		H^s(R^{\conj}, \mathcal{O}\{t\})\Rightarrow \pi_{2t-s}\THH(R).
	\]
	follows from \cite[Theorem 1.8]{pstragowskiPerfectEvenModules2023}.

	Finally we show the claim about when $R$ is furthermore a discrete $p$-complete quasisyntomic commutative ring such that $\THH(R/MU)$ has bounded $p$-power torsion. By asummption that $R$ is a discrete commutative ring, $R$ admits an $\E_{infty}$-$\MU$-algebra structure. There exists a quasiregular semiperfectoid $S$ and cover $R\rightarrow S$ in the quasisyntomic topology as explained in \cite[Proof of Theorem 5.0.3] {hahnMotivicFiltrationTopological2022}. Both of the induced maps $\THH(R; \Z_p^{\wedge})\rightarrow \THH(R/\MU)_p^{\wedge}$ and $\THH(R; \Z_p^{\wedge})\rightarrow \THH(S); \Z_p^{\wedge})$ are $p$-completely eff. 
	By assumption $\THH(R/\MU)_p^{\wedge}$ has bounded $p$-power torsion, and $\THH(S); \Z_p^{\wedge})$ has bounded $p$-power torsion due to $S$ being quasiregular semiperfectoid.
	Thus two faithfully flat descent statements show that $\fil_{\ev}\THH(R)_p^{\wedge}$ the natural comparison map of filtered $\E_{\infty}$-rings into the limit
	\[
		\lim_{(n,m) \in \Delta\times \Delta}\tau_{\geq \ast}
		\THH(R/\MU^{\otimes^{n+1}} \otimes_{\THH(R; \Z_p^{\wedge}} \THH(S^{\otimes_R{m+1}}; \Z_p^{\wedge}) 
	\]
	where, as the notation suggests, the two limits commute so that the order in which they are computed doesn't change the resulting limit. By \cite[Theorem 1.12]{bhattmorrowscholzeTopologicalHochschildHomology2019} we have the equivalence of graded (discrete) rings
	\[
		\pi_{2\ast}\THH(S; \Z_p^{\wedge})
		\simeq
		\bigoplus_{i\geq 0} \gr_{\mathcal{N}^i}\widehat{\Delta}_S\{i\}
	\]
	where $\{i\}$ denotes the Breuil-Kisin twist. By \cite[Definition 5.8.5]{bhattlurieAbsolutePrismaticCohomology2022}, we can put $\Delta_S$ in place of $\widehat{\Delta}_S$ due to being Nygaard complete and therefore obtain the equivalence
	\[
		[\Spf(\pi_{\ast}(\THH(S; \Z_p^{\wedge})))/\G_m]
		\simeq
		(\Spf(R)^{\mathcal{N}})_{s=0}
	\]
	given by the Rees description of prismatization \cite[Section 5.5]{bhattPrismaticFgauges2022}. Thus we obtain $(\Spf(R)^{\mathcal{N}})_{s=0}$ by stackifying the $\Delta\times \Delta$-indexed diagram of rings to obtain a $\Delta^{\op}\times\Delta^{op}$-indexed diagram of schemes and computing the two colimits in one of the two possible orders. Doing so in the other way gives $\widehat{R^{\conj}}_p$ as desired.
	\end{proof}

\subsection{Conjugate Stack Calculations of Theorem B}\label{subsec:thm_b_calc}

\begin{notation}
We establish notation following \cite[Example 7.6.5]{devalapurkarSpherochromatismRepresentationTheory2025}. 
Let $\mathbf{W}$ denote the integral Witt group scheme, and let $\mathbf{W}(-\ast)$ denote the graded lift of $\mathbf{W}$ where the $i$-th coordinate has weight $i$. 
Likewise, we let $\mathrm{W}$ denote the $p$-typical group scheme and let $\mathrm{W}(-\ast)$ denote the graded lift of $\mathrm{W}$ where the ${p^i}$-th coordinate has weight $p^i$.

We let $\widehat{\G_a}$ denote the additive formal group over $\F_p$, $\G_m$ denote the multiplicative group over $\F_p$, $\underline{C}^k(\widehat{\G_a}, \G_m)$ denote the group schemes introduced in \cite[Section 3]{andohopkinsstricklandEllipticSpectraWitten2001}.
\end{notation}

For $0\leq k \leq 3$ that by \cite{andohopkinsstricklandEllipticSpectraWitten2001} we have identifications $\Spec(\F_p[\BU \langle 2k \rangle]) \cong C^k(\widehat{\G_a}, \G_m)$ of group schemes.

\begin{lemma}\label{lem:equ}
	Then the maps of graded affine commutative group schemes over $\Spec(\F_p)$
	\begin{align*}
		\Spec(\THH_{\ast}(\F_p/\MU_{(p)})) &\leftarrow \Spec(\pi_{\ast}R[\BU])\\
		&\xrightarrow{f_1} C^1_{\F_p}(\widehat{\G_a}, \G_m)\\
		&\xleftarrow{\exp} \W_{\F_p}(-\ast)\\
	\end{align*}
	are equivalences. Note $f_1$ denotes the cocycle classification map and $\exp$ denoes the Artin-Hasse exponential. Furthermore 
	Furthermore for all $(p,n)$ other than $(p,n)=(2,2)$ we have
	\begin{align*}
		\Spec(\THH_{\ast}(\Z/p^n/\MU_{(p)})) &\leftarrow \Spec(\pi_{\ast}\Z/p^n[B\left(U(\alpha(p,n)\right)])\\
		&\xrightarrow{f_1} C^1_{\F_p}(\widehat{\G_a}, \G_m) \times \G_{a, \Z/p^n}^{\sharp}\\
		&\leftarrow \W_{\F_p}(-\ast)\times \G_{a, \Z/p^n}^{\sharp}\\
	\end{align*}
	and finally
	\begin{align*}
		\Spec(\THH_{\ast}(\Z_{(p)}/\MU_{(p)})) &\leftarrow \Spec((\Z_{(p)})_{\ast}[\BSU])\\
		&\xrightarrow{f_1} C^2(\widehat{\G_a}, \G_m)\times \Spec(\Z_{(p)})\\
		&\leftarrow F_{\ast}\mathbf{W}(-\ast)_{\Z/p^n}
	\end{align*}

\end{lemma}
\begin{proof}
	Our work of the bialgebric $\THH$ of Thom spectra proves the first arrows in these equivalences. The middle arrows follow from \cite{andohopkinsstricklandEllipticSpectraWitten2001}. The last arrows are classical.
\end{proof}
\begin{corollary}\label{cor:bialg_thh_F_p_MU}
	 
	We have equalities
	\[
		\THH_{\ast}(\F_p/\MU)\cong \mathrm{H}_{\ast}(\BU, \F_p)\cong \mathcal{O}_{\mathrm{C}^1(\widehat{\G_a}, \G_m)}
	\]
	of graded commutative $\F_p$-bialgebras claimed in \cite[pg. 269]{devalapurkarSpherochromatismRepresentationTheory2025}
\end{corollary}

\begin{lemma}
	The Quillen idempotent $\MU\rightarrow\BP \rightarrow \MU$ induces on
	\[
		\Spec\left(\THH_{\ast}(\Z_{(p)}/\MU)\right)
		\cong 
		\Spec\left(\pi_{\ast}(\Z_{(p)}\otimes_{\THH(\MU)})\THH(\MU /MU^{\otimes 2})\right)
		\cong
		F_{\ast}W \times \prod_{(n,p)=1} W
	\]
	the projection
	\[
		F_{\ast}W \times \prod_{(n,p)=1} W
		\xrightarrow{\id_{F_{\ast}W}\times 0}
		F_{\ast}W \times \prod_{(n,p)=1} W
	\]
	onto $F_{\ast}W$.
\end{lemma}

\begin{proof} 
	This follows from tracking where the elements $\sigma^2 v_i$ are sent and where the other elements $\sigma^2 x_i$ are sent by the maps of Lemma~\ref{lem:equ}.
\end{proof}

\begin{corollary}
	The Cartier idempotent $\mathbf{W}\rightarrow \mathbf{W}$ splitting off $W$ determines equivalnces
	\begin{align*}
		\Spec(\THH_{\ast}(\F_p/\BP))
		&\xrightarrow{f_1} 
		C^1_{\F_p, \text{typ}}(\widehat{\G_a}, \G_m)\\
		&\leftarrow
		\W_{\F_p}(-\ast)\\
	\end{align*}
	of graded affine group schemes over $\Spec(\F_p)$. It also determines equivalences
	\begin{align*}
		\Spec(\THH_{\ast}((\Z_{(p)}/\BP))
		&\xrightarrow{f_2} 
		C^2_{\Z_{(p)}, \text{typ}}(\widehat{\G_a}, \G_m)\\
		&\leftarrow 
		F_{\ast}\W_{\Z_{(p)}}(-\ast)\\
	\end{align*}
	of graded affine group schemes over $\Spec(\Z_{(p)})$.
\end{corollary}
We call $C^1_{\F_p, \text{typ}}(\widehat{\G_a}, \G_m)$ and $C^2_{\Z_{(p)}, \text{typ}}(\widehat{\G_a}, \G_m)$ the group schemes cut out of $C^1(\widehat{\G_a}, \G_m)$ and $C^2(\widehat{\G_a}, \G_m)$, respectively.
\begin{lemma}\label{lem:right_unit}
	The map
	\[
		\eta_R^{\BP}:\THH(\Z_{(p)}/\BP)
		\rightarrow 
		\THH(\Z_{(p)}/\BP^{\otimes 2})
		\cong
		\THH(\Z_{(p)}/\BP) \otimes_{\Z_{(p)}} \Z_p \otimes_{\THH(\BP)} \BP
	\]
	satisfies
	\[
		\eta_R^{\BP}(\sigma^2 v_i)=\sigma^2 v_i \otimes 1 +p \otimes \sigma^2 t_i
	\]
	for all $i>0$.
\end{lemma}
\begin{proof}
	We write $\BP_{\ast}\BP\cong \BP[t_1, t_2, \dots]$ with $\abs{t_i}=2p^i-2$
	Recall, say from \cite{ravenelComplexCobordismStable1986}, that
	\[
		p\ell_n=\sum_{i=1}^{n-1}=\ell_iv_{n-i}^{p^i}
	\]
	and
	\[
		\eta_R^{\BP}(\ell_n)=\sum_{i=0}^n \ell_i t_{n-i}^{p^{n}}
	\]
	so that $\eta_R(v_i)\equiv v_i+pt_i$ mod $I^2$ for $I$ denoting the ideal $\ker(\BP_{\ast}\BP\rightarrow \Z_{(p)}$. The result follows from peroperties of suspesion \cite[Section 2]{leeIntegralTopologicalHochschild2026} and that $I/I^2$ is torsion-free.
\end{proof}

\begin{proposition} 
	The Hopf algebroid
	\[\begin{tikzcd}
		{\THH_{2\ast}(\F_p/\BP)} & {\THH_{2\ast}(\F_p/\BP^{\otimes 2})}
		\arrow["{\eta_R}"', shift right, from=1-1, to=1-2]
		\arrow["{\eta_L}", shift left, from=1-1, to=1-2]
		\end{tikzcd}\]
	gives the groupoid scheme
	\[\begin{tikzcd}
		W &&& {W\times F_{\ast}W}
		\arrow["{w_1\leftarrow(w_1, w_2)}"', shift right, from=1-4, to=1-1]
		\arrow["{w_1+V(w_2)\leftarrow(w_1, w_2)}", shift left, from=1-4, to=1-1]
	\end{tikzcd}\]
	so that $\eta_L$ is projection and $\eta_R$ is $(V \otimes \id)\Delta_{\W}$ for $\Delta_{\W}$ dentoe the coproduct.
\end{proposition}

\begin{proof}
	This follows from Lemma~\ref{lem:right_unit}.
\end{proof}

At this point we've proven Theorem B part (ii) for $\BP\langle -1\rangle=\F_p$.

\begin{corollary}[Corollary \ref{intro_bok}]\label{cor:bok_stack}
	We have
	\[
		F_p^{\conj}
		\cong
		\G_a/\G_m
	\]
	as stacks over $\Spec(\F_p)/\G_m$ and that
	\[
		\bigoplus_{t \in \Z} R\Gamma(\F_p^{\conj}, \mathcal{O}\{t\})[2t]
		\cong
		\F_p[x],
	\]
	with $\abs{x}=2$, as graded $\F_p$-modules. Thus the fact that $\F_p^{\conj}\cong \G_a/\G_m$ as stacks over $\Spec(\F_p)/\G_m$ is a ``stacky" version of B\"oksdedt periodicity.
\end{corollary}

\begin{lemma}\label{lem:quot}
	We have
	\[
		[\G_a / \W[F^n]] \cong [F_{\ast}^n\W / F_{\ast}\W]
	\]
	of quotient stacks.
\end{lemma}

\begin{proof}
	We copy the proof of \cite[Proposition 3.5.1]{drinfeldPrismatization2024} but for $F^m$ for $m$ a positive integer in place of $F$. By the exactness of the sequence
	\[
		0
		\rightarrow
		F_{\ast}W
		\xrightarrow{V}
		W
		\xrightarrow{w_0}
		\G_a
		\rightarrow
		0
	\]
	we get $\G_a=[W/F_{\ast}W]$ thus
	\[
		[G_a / W[F^m]] 
		\cong
		[[W/F_{\ast}W]/W[F^m]].
	\]
	Taking quotient stakcs is associative for commuting actions. In paticular, the Verschiebung action of $\F_{\ast}W$ on $W$ commutves with with shift action of $W[F^n]$ on $W$ therefore
	\[
		[[W/F_{\ast}W]/W[F^m]]
		\cong
		[[W/W[F^n]]/F_{\ast}W].
	\]
	Finally, the exacntess of the sequence
	\[
		0
		\rightarrow
		W[F^m]
		\rightarrow{V}
		W
		\rightarrow{F^n}
		F_{\ast}n W
		\rightarrow
		0
	\]
	gives $F_{ast}^n W\cong [W/W[F^m]]$ thereofer
	\[
		[[W/W[F^n]]/F_{\ast}W] \cong [F_{\ast}^nW/F_{\ast}W]
	\]
	as desired. Note $F^{n}V=pF^{n-1}$ because of $FV=p$.
\end{proof}

\begin{lemma}
	Let $A$ denote a graded $\F_p$-Hopf subalgebra of $\W$. the same dimension as $\F_p$-module in each graded piece as the polynomial $\F_p$-algebra $\F[x_n, x_{n+1}, x_{n+2}, \dots]$ with $\deg(i)=p^i$ and $n$ fixed then $A$ is isomorphic to $\mathcal{O}_{F^{n}_{\ast}}$ as a graded $\F_p$-Hopf subalgebra of $\mathcal{O}_{\W}$.
\end{lemma}

\begin{proof}
	Since $A$ is $\F_p$-Hopf subalgebra of $\W$ there exists a $\F_p$-Hopf algebra map $\mathcal{O}_{\W^{\vee}}\rightarrow A^{\wedge}$ which is the same as the $\F_p$-Hopf algebra map $\mathcal{O}_{\W}\rightarrow A^{\wedge}$ becuase $\W$ is self-dual. 
	By assumption of the dimensions of the graded pieces of $A$, $\W\rightarrow A^{\wedge}$ factors through $\mathcal{O}\W/I\rightarrow A^{\wedge}$ for $I=(x_1, \dots, x_{n-1})$. 
	The resulting map $\mathcal{O}/I\rightarrow A^{\wedge}$ is an isomoprhism in each graded part therefore is an isomorphism of $\F_p$-Hopf algebras. 
	This concludes the proof of the uniqueness of the graded $\F_p$-Hopf subalgebra strcture of $A$, and the exists such an $A$ given by $\mathcal{O}_{F^{n}_{\ast}}$.
\end{proof}

\begin{corollary}\label{cor:hopf}
	The graded $\F_p$-Hopf subalgebra $\THH_{\ast}(\BP\langle n \rangle/\BP)$ of $\THH_{\ast}(\F_p/\BP)$ induces the faithfully flat map
	\[
		F^{n_1}_{\ast}:\W\rightarrow F^{n+1}_{\ast}
	\]
	given by the $(n+1)$-fold Frobenius.
\end{corollary}

\begin{lemma}\label{lem:F}
	The $\E_3$-$\MU$-algebra map $\BP\langle n \rangle \rightarrow \BP\langle -1\rangle\cong \F_p$ of \cite{hahnwilsonRedshiftMultiplicationTruncated2022} induces the $\F_p$-Hopf subalgebra
	\[
		\THH_{\ast}\left(\BP\langle n \rangle / \BP\right)/(p,v_1, \dots, v_n)
		\cong 
		\mathcal{O}_{F^{n+1}_{\ast}\W}
	\]
	structure of
	\[
		\Spec(\THH_{\ast}(\F_p/\BP))
		\cong
		\mathcal{O}_{\W}
	\]
	that is equivalent to the faithfully flat graded affine group scheme map
	\[
		F^{n+1}:\W \rightarrow F^{n+1}_{\ast}\W.
	\]
	over $\Spec(\F_p)$.
\end{lemma}
\begin{proof}
	By Corollary~\ref{cor:hopf}, we know there is an isomoprhism of graded affine group schemes between
	\[
		\Spec\left(\THH_{\ast}\left(\BP\langle n \rangle / \BP\right)/(p,v_1, \dots, v_n)\right)
	\]
	and $F^{n+1}_{\ast}\W$. Let $f$ denote the subalgebra map. We have that
	\[
		f(\sigma^2 v_i)
		\equiv
		\sigma^2 v_{i+n}
		\mod
		\text{ units }
		\F_p^{\times}
	\]
	for all $i\geq 0$ by primitiveness and degree considerations. 
\end{proof}

We prove Theorem \ref{thm:B} part (ii) whose statement we now recall. 
We have
\[
	\BP\langle n \rangle^{\conj}_{p=\cdots=v_n=0}
	=
	[F^{m+1}_{\ast}\W_{\Z_p} / F_{\ast}\W_{\Z_p}]/\G_m
	=
	[\G_a / \W[F^n]]\G_m
\]
as affine group schemes over $\Spec(\Z_p[v_1,\dots, v_n])/\G_m$ where the quotient is associated to $pF^n:F_{\ast}\W_{\F_p}\rightarrow F_{\ast}^{n+1}\W_{\F_p}$.

\begin{proof}[Proof of Theorem B part (ii)]
	By Lemma~\ref{lem:F} we have that $\BP\langle n \rangle^{\conj}_{p=\cdots=v_n=0}$ is the graded quotient stack associated to the graded affine groupoid
	\[\begin{tikzcd}
		{F_{\ast}^{n+1}\W} && {F_{\ast}^{n+1}\W \times \F_{\ast}W}
		\arrow["{\id_{F_{\ast}^{n+1}}\times 0}"', shift right, from=1-3, to=1-1]
		\arrow["0\times pF^n", shift left, from=1-3, to=1-1]
	\end{tikzcd}\]
	By Lemma~\ref{lem:quot} we have an equivalence
	\[
		[\G_a / \W[F^n]] \simeq [F_{\ast}^n\W / F_{\ast}\W].
	\]
	of quotient stacks.
\end{proof}

\begin{corollary}[Corollary~\ref{intro:BPn}]\label{cor:BPn}
	We have
	\[
		\Spec\left(\THH_{\ast}(\BP\langle n \rangle)/(p, \dots, v_n)\right)
		\cong
		\BP\langle n \rangle^{\conj}_{p=v_1=\cdots=v_n=0}
		\cong
		[F^{n+1}_{\ast}W/F_{\ast}W]
	\]
	as graded affine group schemes over $\Spec(\F_p)$ and that
	\[
		\bigoplus_{t \in \Z} R\Gamma(\BP^{\conj}_{p=v_1=\cdots=v_n=0}, \mathcal{O}\{t\})[2t]
		\cong
		\F_p[\mu_{n+1}]\otimes \Lambda_{\F_p}(\lambda_1, \dots, \lambda_{n+1}),
	\]
	with $\abs{\mu_{n+1}}=2p^{n+1}$ and $\abs{\lambda_i}=2p^i-1$, as graded $\F_p$-algebras. This is a ``stacky" version of the calculation
	\[
		\pi_{\ast}\left(\gr_{\mot}^{\ast}\THH(\BP\langle n \rangle)\right)/(p, v_1, \dots, v_n)
		\cong
		\F_p[\mu_{n+1}]\otimes \Lambda_{\F_p}(\lambda_1, \dots, \lambda_{n+1})
	\]
	of \cite{angelini-knollSyntomicCohomologyTruncated2026}.
\end{corollary}

\begin{proof}
	Consider the isomoprhism of graded schemes
	\[
		\mathbb{A}^1[2p^{n+1}] \times F_{\ast}^{n+1}\W\rightarrow F^{n}_{\ast}\W
	\]
	over $\Spec(\F_p)$ induced by taking Teichm\"uller representative and Verschiebung. This gives
	\[
		[F_{\ast}^m/F_{\ast}W]/G_m
		\cong
		\mathbb{A}^1[2p^{n+1}] \times BW[F^{m+1}]
	\]
	Since $R\Gamma(\mathbb{A}^1, \mathcal{O}{2p^{n+1}})\cong \F_p[t]$, it remains to argue
	\[
		\bigoplus_{t \in \Z} R\Gamma(B\W[F^m], \mathcal{O}\{t\})[2t]
		\cong
	 	\Lambda_{\F_p}(\lambda_1, \dots, \lambda_{n+1}).
	\]
	The Cartier dual $W[F^n]^{\vee}$ of $W[F^n]$ is $\widehat{W_n}$, the completion of the $n$-truncated Witt vector group scheme $W_n:=W/V^m$ at the orgin \cite[Proposition C.1]{devalapurkarTopologicalHochschildHomology2023}, therefore we will study $R\Gamma(B\W[F^m], \mathcal{O}\{t\})$ in terms of quasicoherent sheaves on $\widehat{W_n}$. 
	In fact by \cite[Proposition C.4]{devalapurkarTopologicalHochschildHomology2023}, there exists a symmetric monoidal (with respect to the convolution tensor product on $\QCoh(W_n)$) fully faithful functor
	\[
		\Phi: \QCoh(BW[F^n]\hookrightarrow \QCoh(\widehat{W_n})
	\]
	so that under the composition
	\[
		\QCoh(BW[F^n])
		\xrightarrow{\Phi}
		\QCoh(\widehat{W_n})
		\rightarrow
		D(\mathcal{O}_{BW[F^n]})
	\]
	we have $\Phi(\mathcal{O}_{BW[F^n]})=\F_p$. Thus
	\begin{align*}
		R\Gamma(BW[F^n], \mathcal{O})
		&\cong
		R\Hom_{\QCoh(BW[F^n])}(\mathcal{O}, \mathcal{O})\\
		&\cong
		R\Hom_{\QCoh(BW[F^n])}(\Phi(\mathcal{O}), \Phi(\mathcal{O}))\\
		&=\Hom_{D(\mathcal{O}_{BW[F^n]})}(\F_p, \F_p).
	\end{align*}
	Finally one concludes
	\[
		\bigoplus_{t \in \Z} R\Gamma(B\W[F^m], \mathcal{O}\{t\})[2t]
		\cong
	 	\Lambda_{\F_p}(\lambda_1, \dots, \lambda_{n+1}).
	\]
	by, for example, resolving $\F_p$ in $D(\mathcal{O}_{BW[F^n]})$ by the Koszul resolution.
\end{proof}